\documentclass{amsart}

\usepackage{amsfonts}
\usepackage{amsmath}
\usepackage{amssymb}
\usepackage{dsfont}
\usepackage{amscd}
\usepackage{amsthm}

\newtheorem{theorem}{Theorem}[section]
\newtheorem{lemma}[theorem]{Lemma}
\theoremstyle{definition}
\newtheorem {definition}[theorem]{Definition}
\newtheorem{coro}[theorem]{Corollary}
\theoremstyle{remark}
\newtheorem{remark}[theorem]{Remark}

\def\ra{\rightarrow}
\def\iy{\infty}

\def\be{\begin{equation}}
\def\ee{\end{equation}}
\def\ba{\begin{eqnarray*}}
\def\ea{\end{eqnarray*}}
\def\bae{\begin{eqnarray}}
\def\eae{\end{eqnarray}}
\def\bc{\begin{center}}
\def\ec{\end{center}}
\def\|J|{\parallel J\parallel}

\def\Pnk{\mathfrak{B}_{n,p}}
\def\Pnki{\mathfrak{B}_{n,p,i}}

\def\rs{\emph{reduction} step}
\def\ds{\emph{dominating set }}
\def\var{\mathrm{Var}}

\begin{document}
\title[Fluctuation of Eigenvalues for Random Toeplitz  Matrices]{Fluctuation of Eigenvalues for Random Toeplitz and Related Matrices}

%\date{September, 2009}
\author[ D.-Z. Liu X. Sun and Z.-D. Wang]{ Dang-Zheng Liu, Xin Sun and Zheng-Dong Wang}

\address{School of Mathematical Sciences, Peking University, Beijing,
100871, P.R. China} \email{  dzliumath@gmail.com} \email{
xin.sun@hotmail.com}
 \email{ zdwang@pku.edu.cn}

\maketitle

\begin{abstract}

Consider random symmetric Toeplitz matrices $T_{n}=(
a_{i-j})_{i,j=1}^{n}$ with matrix entries $a_{j}, j=0,1,2,\cdots,$
being independent real  random variables
 such that

\be \nonumber \mathbb{E}[a_{j}]=0, \ \ \mathbb{E}[|a_{j}|^{2}]=1 \ \
\textrm{for}\,\ \ j=0,1,2,\cdots,\ee (homogeneity of 4-th moments)
\be{\nonumber\kappa=\mathbb{E}[|a_{j}|^{4}],}\ee
 \noindent and further (uniform boundedness)\be\nonumber\sup\limits_{j\geq 0} \mathbb{E}[|a_{j}|^{k}]=C_{k}<\iy\ \ \
\textrm{for} \ \ \ k\geq 3.\ee 

Under the assumption of  $a_{0}\equiv 0$, we prove a central limit
theorem for linear statistics of eigenvalues for a fixed polynomial
with degree $\geq 2$. Without the assumption, the CLT can be easily modified to a possibly non-normal limit law. In a special case where  $a_{j}$'s are
Gaussian, the result has been obtained by Chatterjee for some test
functions. Our derivation is based on a simple trace formula for
Toeplitz matrices and fine combinatorial analysis. Our method can
apply to other related random matrix models, including Hankel
matrices and product of several Toeplitz matrices in a flavor of
free probability theory etc. Since Toeplitz matrices are quite
different from the Wigner and Wishart matrices, our results enrich
this topic. 
\end{abstract}



\section{Introduction and main results}
\setcounter{equation}{0}

Toeplitz matrices emerge in many aspects of mathematics and physics
and also in plenty of applications, see Grenander and
Szeg$\ddot{o}$'s book \cite{GS} for
 a detailed introduction to deterministic Toeplitz matrices.
   The study of random Toeplitz matrices with
independent entries is proposed by Bai in his review paper
\cite{bai}. Since then,  the literature around the asymptotic
distribution of eigenvalues for random Toeplitz and related matrices
is very large,  including the papers of  Basak and Bose \cite{bb},
Bose et al. \cite{bcg, bm}, Bryc et al. \cite{bdj}, Hammond and
Miller \cite{hm}, Kargin \cite{kargin}, Liu and Wang \cite{LW},
Massey et al. \cite{mms}, etc. We refer to  \cite{bb} for recent
progress. However, the study of fluctuations of eigenvalues for
random Toeplitz matrices is quite little,  to the best   of our
knowledge, the only known result comes from Chatterjee
\cite{chatterjee} in the special case where the matrix entries are
Gaussian distributions. In this paper we will derive  a central
limit theorem (CLT for short) for linear statistics of eigenvalues
of random Toeplitz and related matrices, for some of which the
asymptotic distributions of eigenvalues are also new, including
sparse Toeplitz and Hankel matrices and singular values of powers of
Toeplitz matrices.

In the literature fluctuations of eigenvalues for random matrices
have been extensively  studied.  The investigation of central limit
theorems for linear statistics of eigenvalues of random matrices
dates back to the work of Jonsson~\cite{jonsson} on Gaussian Wishart
matrices. Similar work for the Wigner matrices was obtained by Sinai
and Soshnikov \cite{SS}. For further discussion on Wigner (band)
matrices and Wishart matrices and their generalized models, we refer
to Bai and Silverstein's book \cite{BS}, recent papers
\cite{AZ,chatterjee} and the references therein. For another class
of invariant random matrix ensembles, Johansson \cite{johansson3}
proved a general result which implies  CLT for linear statistics of
eigenvalues. Recently, Dumitriu and Edelman \cite {DE} and Popescu
\cite{popescu} proved that CLT holds for tridiagonal random matrix
models.

Another important contribution is the work of Diaconis et al.
\cite{Dsha,DEvans}, who proved similar results for random unitary
matrices. These results are closely connected to Szeg$\ddot{o}$'s
limit theorem (see \cite{johansson1}) for the determinant of
Toeplitz matrices with $(j,k)$ entry $\widehat{g}(j-k)$,  where
$\widehat{g}(k)=\frac{1}{2\pi}\int_{0}^{2\pi}g(e^{i\theta})e^{-i
k\theta} d\,\theta$, see \cite{diaconis, basor} and references
therein for connections between random matrices and Toeplitz
determinants. Here we emphasize that Szeg$\ddot{o}$'s limit theorem
implies a CLT.

Now we turn to our model.
 The matrix of the form
$T_{n}=(a_{i-j})_{i,j=1}^{n}$ is called a Toeplitz matrix. If we
introduce the Toeplitz or Jordan matrices
$B=(\delta_{i+1,j})_{i,j=1}^{n}$ and
$F=(\delta_{i,j+1})_{i,j=1}^{n}$, respectively called the ``backward
shift" and ``forward shift" because of their effect on the elements
of the standard basis $\{e_{1}, \cdots, e_{n}\}$ of
$\mathbb{R}^{n}$, then an $n \times n$ matrix $T$ can be written in
the form \be T=\sum_{j=0}^{n-1}a_{-j} B^{j}+ \sum_{j=1}^{n-1}a_{j}
F^{j}\ee if and only if $T$ is a Toeplitz matrix where $a_{-n+1},
\cdots, a_{0}, \cdots,a_{n-1}$ are complex numbers \cite{hj}. It is
worth emphasizing that this representation of a Toeplitz matrix is
of vital importance as the starting point
 of our method. The
``shift" matrices $B$ and $F$ exactly present the information of the
traces.

Consider a Toeplitz band matrix as follows. Given a band width
$b_{n}<n$, let \be{\eta_{ij}=
 \begin{cases} 1,\ \ |i-j|\leq b_{n};\\
0,\ \ \text{otherwise.}
\end{cases}
}\ee Then a Toeplitz band matrix is
\be{\label{bandtmatrices}T_{n}=(\eta_{ij}\,
a_{i-j})_{i,j=1}^{n}.}\ee Moreover, the Toeplitz band matrix $T_{n}$
can also be rewritten in the form

\be \label{basicrepresentation}T_{n}=\sum_{j=0}^{b_{n}}a_{-j} B^{j}+
\sum_{j=1}^{b_{n}}a_{j} F^{j}=
a_{0}I_{n}+\sum_{j=1}^{b_{n}}\left(a_{-j} B^{j}+a_{j}
F^{j}\right),\ee where $I_{n}$ is the identity matrix. Obviously, a
Toeplitz matrix can be considered as a band matrix with the
bandwidth $b_{n}=n-1$. In this paper, the basic model under
consideration consists of $n \times n$ random symmetric Toeplitz
band matrices $T_{n}=(\eta_{ij}\, a_{i-j})_{i,j=1}^{n}$ in
Eq.\,(\ref{bandtmatrices}). We assume that $a_{j}=a_{-j}$ for
$j=1,2,\cdots$, and $\{a_{j}\}_{j=1}^{\iy}$ is a sequence of
independent real  random variables %(implying that all odd moments
%vanish)
 such that

\be\label{symmetrictoeplitz1} \mathbb{E}[a_{j}]=0, \ \
\mathbb{E}[|a_{j}|^{2}]=1 \ \ \textrm{for}\,\ \ j=1,2,\cdots,\ee
(homogeneity of 4-th moments)
\be{\label{4thmoment}\kappa=\mathbb{E}[|a_{j}|^{4}],}\ee
 \noindent and further (uniform boundedness)\be\label{symmetrictoeplitz2}\sup\limits_{j\geq 1} \mathbb{E}[|a_{j}|^{k}]=C_{k}<\iy\ \ \
\textrm{for} \ \ \ k\geq 3.\ee %Notice that Chatterjee
%\cite{chatterjee} has pointed,
In addition, we also assume $a_{0}\equiv 0$ (we will explain in  Remarks \ref{diagonalentry} and \ref{remarka0} below!) and  the bandwidth $b_{n}\ra
\iy$ but $b_{n}/n
\rightarrow b\in [0,1]$ as $n\rightarrow \infty$. % Note that in
%order to obtain Gaussian fluctuation of random Toeplitz
% matrices, we must  whose elements with general random (we will explain why we do, see below!)

 Set $A_{n}=\frac{T_{n}}{\sqrt{b_{n}}}$, a linear statistic of  eigenvalues $\lambda_{1}, \cdots, \lambda_{n}$ of $A_{n}$
 is a function of the form
 \be
\frac{1}{n}\sum_{j=1}^{n}f(\lambda_{j}),
 \ee
where $f$ is some fixed function. In particular, when $f(x)=x^{p}$
we write\be \omega_{p}
=\frac{\sqrt{b_{n}}}{n}\sum_{j=1}^{n}\left(\lambda_{j}^{p}-\mathbb{E}[\lambda_{j}^{p}]\right).
 \ee
These $\omega_{p}$, $p=2,3, \ldots,$ are our main objects. Note that
$\frac{1}{n}\sum_{j=1}^{n}\left(\lambda_{j}^{p}-\mathbb{E}[\lambda_{j}^{p}]\right)$
converges weakly to zero as $n \longrightarrow \infty$, moreover
under the condition $$\sum_{j=1}^{\iy}\frac{1}{b^{2}_{n}}<\iy$$ we
have a strong convergence, see \cite{bb, kargin, LW}. We remark that
the fluctuations \be
=\frac{1}{n}\sum_{j=1}^{n}\lambda_{j}^{p}-\frac{1}{n}\sum_{j=1}^{n}\mathbb{E}[\lambda_{j}^{p}].
 \ee
are of order $\frac{1}{\sqrt{b_{n}}}$ while those for Wigner
matrices are of order $\frac{1}{n}$, more like the case of classical
central limit theorem. This shows that the correlations between
eigenvalues for Toeplitz matrices are much weaker  than those for
Wigner matrices (another different phenomenon is that the limiting
distribution for random Toeplitz matrices has unbounded support, see
\cite{bdj, hm}). The potential reasons for this phenomenon may come
from the fact that the order of the number of independent variables
is $O(n)$ for Toeplitz matrices while it is $O(n^{2})$ for Wigner
matrices. On the other hand, the eigenvalues of random Toeplitz
matrices are obviously not independent, so it is different from the
case of CLT for independent variables.

In special case that the matrix entries $a_{j}$ are Gaussian
distributions, by using his notion
  of `second order Poincar\'e inequalities' Chatterjee in \cite{chatterjee}
  proved
the following theorem:

\vspace{\medskipamount} \noindent{\bf Theorem} (\cite{chatterjee},
Theorem 4.5) {\it Consider the Gaussian Toeplitz matrices
$T_{n}=(a_{i-j})_{i,j=1}^{n}$, i.e. $a_{j}=a_{-j}$ for
$j=1,2,\cdots$, and $\{a_{j}\}_{j=0}^{\iy}$ is a sequence of
independent standard Gaussian random variables. Let $p_{n}$ be a
sequence of positive integers such that $p_{n} = o(\log n/\log \log
n)$. Let $A_{n}=T_{n}/\sqrt{n}$, then, as $n \ra \iy$,
\[
\frac{\mathrm{tr} (A_{n}^{p_{n}}) - \mathbb{E}[\mathrm{tr}
(A_{n}^{p_{n}})]}{\sqrt{\var\left(\mathrm{tr}
(A_{n}^{p_{n}})\right)}} \ \text{ converges in total variation to }
\ N(0,1).
\]The  $\mathrm{CLT}$ also holds for $\mathrm{tr} (f(A_{n}))$, when $f$ is a
fixed nonzero polynomial with nonnegative coefficients. }
\vspace{\medskipamount}

The author remarked that
  the above theorem is only for Gaussian Toeplitz matrices based on the obvious fact: considering the function
$f(x)=x$,    CLT may not  hold for linear statistics of
non-Gaussian Toeplitz matrices.  The author also  remarked that the
above theorem says nothing about the limiting formula of the
variance $\var\left(\mathrm{tr} (A_{n}^{p_{n}})\right)$. However, we
assert that CLT holds for a test function  $f(x)=x^{2p}$ even for
non-Gaussian Toeplitz matrices. When $f(x)=x^{2p+1}$ the fluctuation is Gaussian if and only if the diagonal random variable $a_{0}$ is Gaussian. Moreover, if we suppose $a_{0}\equiv
0$, we can obtain CLT for any fixed polynomial test functions.  On
the other hand, we can calculate the variance in terms of integrals
associated with pair partitions. Unfortunately, our method fails to deal
with the test function $f(x)=x^{p_{n}}$, where $p_{n}$ depends on
$n$.

Our study is inspired by the work of Sinai and Soshnikov \cite{SS},
but new ideas are needed since the structure of Toeplitz matrices is
quite different from that of Wigner matrices. In addition, our
method can apply to other related random matrix models, including
Hermitian Toeplitz matrices, Hankel matrices, sparse Toeplitz and
Hankel matrices, singular values of powers of Toeplitz matrices, and
product of several matrices in a flavor of free probability theory.

Now we state the main theorem of our paper as follows.
\begin{theorem}\label{tpg}
Let $T_{n}$ be  real symmetric
((\ref{symmetrictoeplitz1})--(\ref{symmetrictoeplitz2})) random
Toeplitz band matrix with the bandwidth $b_{n}$, where $b_{n}/n
\rightarrow b \in [0,1]$ but $b_{n}\ra \iy$ as $n\rightarrow
\infty$. Set $A_{n}=T_{n}/\sqrt{b_{n}}$ and
\be{\omega_{p}=\frac{\sqrt{b_{n}}}{n}\left(\mathrm{tr}(A_{n}^{p})-\mathbb{E}[\mathrm{tr}(A_{n}^{p})]\right).}\ee
For every  $p\geq 2$, we have\be \omega_{p}\longrightarrow
N(0,\sigma_p^2)\ee in distribution as $n\rightarrow \infty$.
Moreover, for a given polynomial \be{Q(x)=\sum_{j=2}^{p}q_{j}
x^{j}}\ee with degree $p\geq 2$, set
\be{\omega_{Q}=\frac{\sqrt{b_{n}}}{n}\left(\mathrm{tr}Q(A_{n})-\mathbb{E}[\mathrm{tr}Q(A_{n})]\right),}\ee
we also have\be \omega_{Q}\longrightarrow N(0,\sigma_Q^2)\ee in
distribution as $n\rightarrow \infty$. Here the variances
$\sigma_p^{2}$ and $\sigma_Q^2$ will be given in section
\ref{variance}.
\end{theorem}

From the proof of our main theorem, we can easily derive an
interesting result concerning product of independent variables whose
subscripts satisfy certain ``balance'' condition. When $p=2$, it is
a direct result from the classical central limit theorem. Here we
state it but omit its proof, see Remark \ref{perturbancefactor} for
detailed explanation.

\begin{coro}\label{corollary} Suppose that
$a_{j}=a_{-j}$ for $j=1,2,\cdots$, and $\{a_{j}\}_{j=1}^{\iy}$ is a
sequence of independent random variables satisfying the assumptions
(\ref{symmetrictoeplitz1})--(\ref{symmetrictoeplitz2}).
%Let
%$\{b_{n}\}$ be a positive integer sequence and $b_{n}\ra \iy$ as
%$n\rightarrow \infty$.
For every $p\geq 2$, \be{\frac{1}{n^{\frac{p-1}{2}}}\sum^{n}_{0\neq
j_{1},\ldots,
j_{p}=-n}\left(\prod_{l=1}^{p}a_{j_{l}}-\mathbb{E}[\prod_{l=1}^{p}a_{j_{l}}]\right)
\large{\delta}_{0,\sum\limits_{l=1}^{p}j_{l}}}\ee converges in
distribution to a Gaussian distribution $N(0,\sigma_p^2)$. Here the
variance $\sigma_p^{2}$ as in section \ref{variance} corresponding
to the case $b=0$.
\end{coro}

\begin{remark}[on the diagonal entry $a_{0}$]\label{diagonalentry} Rewrite $A_{n}(a_{0})=A_{n}(0)+\frac{a_{0}}{\sqrt{b_{n}}}I_{n}$, where $A_{n}(0)$ denotes the matrix with $a_{0}=0$. It is easy to see that
\be
\omega_{p}(a_{0})=\omega_{p}(0)+p\frac{\mathrm{tr}(A_{n}^{p-1}(0))}{n}(a_{0}-\mathbb{E}[a_{0}])+O(b_{n}^{-1/2}),
\ee
 which converges in distribution to the distribution of $\sigma_{p}\mathfrak{n}+p M_{p-1}(a_{0}-\mathbb{E}[a_{0}])$. Here the $\mathfrak{n}$ is the standard normal distribution, independent of $a_{0}$, and $M_{p-1}$ is the (p-1)-th moment in Theorem \ref{limitmoment}. Since $M_{p}=0\Longleftrightarrow p \  \mbox{is  odd}$, if  $a_{0}$ is neither a constant a.s. nor Gaussian then the fluctuation  is not Gaussian for odd $p$.
\end{remark}

The remaining part of the paper is arranged as follows. Some
integrals associated with pair partitions are defined in section
\ref{Integrals associated with pair partitions}. Sections
\ref{expectation}, \ref{variance} and  \ref{highermoments} are
devoted to the proof of Theorem \ref{tpg}.  We will extend our main
result
to  other models closely connected with Toeplitz matrices in section \ref{extensions}.%first calculate the variance of the fluctuation in section
%\ref{variance} and then use the same method to estimate higher
%moments in section \ref{highermoments}.

\section{Integrals associated with pair partitions}\label{Integrals associated with pair partitions}
\setcounter{equation}{0} In order to calculate the moments of the
limit distribution and the limiting  covariance matrix of  random
variables $\omega_{p}$, we first review some basic combinatorical
concepts, and  then define some integrals associated with pair
partitions.

\begin{definition} Let the set $[n]=\{1,2,\cdots,n\}$.

(1) We call $\pi=\{V_{1},\cdots,V_{r}\}$ a partition of $[n]$ if the
blocks  $V_{j} \,(1\leq j \leq r)$ are pairwise disjoint, non-empty
subsets of $[n]$ such that $[n]=V_{1}\cup\cdots\cup V_{r}$. The
number of blocks of $\pi$ is denoted by $|\pi|$, and the number of
 elements of $V_{j}$ is denoted by $|V_{j}|$.

(2) Without loss of generality, we assume that $V_{1},\cdots,V_{r}$
have been arranged such that $s_{1}<s_{2}<\cdots<s_{r}$, where
$s_{j}$ is the smallest number of $V_{j}$. Therefore we can define
the projection $\pi (i)=j$ if $i$ belongs to the block $V_{j}$;
furthermore for two elements $p,q$ of $[n]$ we write
$p\thicksim_{\pi} q$ if $\pi (p)=\pi (q)$.

(3) The set of all partitions of $[n]$ is denoted by
$\mathcal{P}(n)$, and the subset consisting of all pair partitions,
i.e. all $|V_{j}|=2$, $1\leq j \leq r$, is denoted by
$\mathcal{P}_{2 }(n)$. %The subset of $\mathcal{P}_{2 }(n)$
%consisting of such pair partitions that each contains exactly one
%even number and one odd number is denoted by $\mathcal{P}^{1}_{2
%}(n)$.
 Note that $\mathcal{P}_{2 }(n)$ is an empty set if $n$ is
odd.

(4) Suppose $p,q$ are positive integers and  $p+q$ is even, we
denote a subset of $\mathcal{P}_{2}(p+q)$ by $\mathcal{P}_{2}(p,q)$,
which consists of such pair partitions $\pi$: there exists $1\leq i
\leq p<j \leq p+q$ such that $i\thicksim_{\pi} j$ (we say that there
is one crossing match in $\pi$).
% When $q=0$, we mean that $\mathcal{P}_{2}(p,0)$ denotes
%$\mathcal{P}_{2}(p)$. In this section we assume $p+q$ is even %unless
%we mention it

(5) When $p$ and $q$ are both  even, we denote a subset of
 $\mathcal{P}(p+q)$ by $\mathcal{P}_{2,4}(p,q)$, which consists
of such  partitions $\pi=\{V_{1},\cdots,V_{r}\}$ satisfying

(i) $|V_{j}|=2$, $1\leq j \leq r, j\neq i$ and  $|V_{i}|=4$ for some
$i$.

(ii)  $V_{j}\subseteq \{1,2, \ldots, p\} \, \mathrm{or}\, \{p+1,
p+2, \ldots, p+q\}$ for $1\leq j \leq r, j\neq i$.

(iii) two elements of $V_{i}$ come from $\{1,2, \ldots, p\}$ and the
other two come from
$\{p+1, p+2, \ldots, p+q\}$. %except for every block has two elements
%except that there exists one block with four elements, and has

For other cases of $p$ and $q$, we     assume
$\mathcal{P}_{2,4}(p,q)$ is an empty set.
\end{definition}

Now  we define several types of definitive integrals associated with
$\pi \in \mathcal{P}_{2}(p,q)$ or $\pi \in \mathcal{P}_{2,4}(p,q)$.
For the reader's  convenience, we suggest to omit them for the
moment  and refer to them when they are needed in sections
\ref{expectation} and \ref{variance}. Let the parameter $b\in
[0,1]$.

First, for $\pi \in \mathcal{P}_{2}(p,q)$ we set
\be{\epsilon_{\pi}(i)=
 \begin{cases} 1,\ \
i \ \text{is the smallest number of}\  \pi^{-1}(\pi(i));\\
-1,\ \ \text{otherwise.}
\end{cases}
}\ee To every pair partition $\pi \in \mathcal{P}_{2}(p,q)$, we
construct a projective relation between two groups of unknowns
${y_{1},\ldots,y_{p+q}}$ and ${x_{1},\ldots,x_{\frac{p+q}{2}}}$ as
follows:
\be{\label{unknownsrelation}\epsilon_{\pi}(i)\,y_{i}=\epsilon_{\pi}(j)\,y_{j}=x_{\pi(i)}
}\ee whenever $i\thicksim_{\pi} j$. Thus, we have an identical
equation \be{\label{identicaleq}\sum_{j=1}^{p+q}y_{j}\equiv 0 }.\ee

For $x_{0}, y_{0}\in [0,1]$ and ${x_{1},\ldots,x_{\frac{p+q}{2}}\in
[-1, 1]}$, we define two kinds of integrals with Type I by
%\prod_{j=1}^{p+q}\chi_{[0,1]}(x_{0}+b\sum_{i=1}^{j}y_{i})
 \begin{align}
 \label{typeI-} & f_{I}^{-}(\pi)=\nonumber\\
& \int_{[0,1]^{2}\times
[-1,1]^{\frac{p+q}{2}}}\delta\left(\sum_{i=1}^{p}y_{i}\right)
\prod_{j=1}^{p}\chi_{[0,1]}(x_{0}+b\sum_{i=1}^{j}y_{i})
\prod_{j'=p+1}^{p+q}\chi_{[0,1]}(y_{0}+b\sum_{i=p+1}^{j'}y_{i})\,d\,y_{0}
\prod_{l=0}^{\frac{p+q}{2}}d\, x_{l}\end{align} and
\begin{align}\label{typeI+} &f_{I}^{+}(\pi)=\nonumber\\
&\int_{[0,1]^{2}\times
[-1,1]^{\frac{p+q}{2}}}\delta\left(\sum_{i=1}^{p}y_{i}\right)
\prod_{j=1}^{p}\chi_{[0,1]}(x_{0}+b\sum_{i=1}^{j}y_{i})
\prod_{j'=p+1}^{p+q}\chi_{[0,1]}(y_{0}-b\sum_{i=p+1}^{j'}y_{i})\,d\,y_{0}
\prod_{l=0}^{\frac{p+q}{2}}d\, x_{l}.
\end{align} Here $\delta$ is
the Dirac function and $\chi$ is the indicator  function. Note that
$\sum_{j=1}^{p}y_{j}\neq 0$ by the definition of
$\mathcal{P}_{2}(p,q)$, therefore the above integrals are multiple
integrals in
$(\frac{p+q}{2}+1)$ variables. %and \be{\tau(i)=
% \begin{cases} 1,\ \
%1\leq i \leq p;\\
%-1,\ \ p<i\leq p+q.
%\end{cases}
%}\ee

Next, for $\pi=\{V_{1},\ldots,V_{\frac{p+q}{2}-1}\}\in
\mathcal{P}_{2,4}(p,q)$ (denoting the block with four elements by
$V_{i}$ ), we set  for $\pi(k)\neq i$ \be{\tau_{\pi}(k)=
 \begin{cases} 1,\ \
k \ \text{is the smallest number of}\  \pi^{-1}(\pi(k));\\
-1,\ \ \text{otherwise }
\end{cases}
}\ee while for $\pi(k)=i$ \be{\tau_{\pi}(k)=
\begin{cases} 1,\ \
k \ \text{is the smallest or largest number of}\  \pi^{-1}(\pi(k));\\
-1,\ \ \text{otherwise}.
\end{cases}}\ee

 To every partition $\pi \in \mathcal{P}_{2,4}(p,q)$, we
construct a projective relation between two groups of unknowns
${y_{1},\ldots,y_{p+q}}$ and ${x_{1},\ldots,x_{\frac{p+q}{2}-1}}$ as
follows:
\be{\label{unknownsrelation}\tau_{\pi}(i)\,y_{i}=\tau_{\pi}(j)\,y_{j}=x_{\pi(i)}
}\ee whenever $i\thicksim_{\pi} j$. Then two kinds of integrals with
Type II are defined respectively  by
 \begin{align}\label{typeII-} f_{II}^{-}(\pi)=
 \int_{[0,1]^{2}\times
[-1,1]^{\frac{p+q}{2}-1}}
\prod_{j=1}^{p}\chi_{[0,1]}(x_{0}+b\sum_{i=1}^{j}y_{i})
\prod_{j'=p+1}^{p+q}\chi_{[0,1]}(y_{0}+b\sum_{i=p+1}^{j'}y_{i})\,d\,y_{0}
\prod_{l=0}^{\frac{p+q}{2}-1}d\, x_{l}\end{align} and
\begin{align}\label{typeII+} f_{II}^{+}(\pi)=
 \int_{[0,1]^{2}\times
[-1,1]^{\frac{p+q}{2}-1}}
  \prod_{j=1}^{p}\chi_{[0,1]}(x_{0}+b\sum_{i=1}^{j}y_{i})
\prod_{j'=p+1}^{p+q}\chi_{[0,1]}(y_{0}-b\sum_{i=p+1}^{j'}y_{i})\,d\,y_{0}
\prod_{l=0}^{\frac{p+q}{2}-1}d\, x_{l}\nonumber.\end{align}

\section{Mathematical expectation}
\label{expectation} \setcounter{equation}{0} In this section,  we
will review some results about  the moments of the limiting
distribution  of eigenvalues in \cite{LW}, for the convenience of
the readers and further discussion.

\begin{theorem}\label{limitmoment}$\mathbb{E}[\frac{1}{n}\mathrm{tr}(A_{n}^{2k})]=M_{2k}+o(1)$
and $\mathbb{E}[\frac{1}{n}\mathrm{tr}(A_{n}^{2k+1})]=o(1)$
 as $n \longrightarrow  \iy$ where \be
\label{bandtoeplitz:moment}M_{2k}=\sum_{\pi\in \mathcal{P}_{2
}(2k)}\int_{[0,1]\times
[-1,1]^{k}}\prod_{j=1}^{2k}\chi_{[0,1]}(x_{0}+b\sum_{i=1}^{j}\epsilon_{\pi}(i)\,x_{\pi(i)})
\prod_{l=0}^{k}\mathrm{d}\, x_{l}.\ee
\end{theorem}

Let us first give a lemma about traces of Toeplitz band matrices.
Although its proof is simple, it is very useful in treating random
matrix models closely related to Toeplitz matrices.

\begin{lemma}\label{toeplitzlemma}
For Toeplitz band matrices $T_{l,n}=(\eta_{ij}\,
a_{l,i-j})_{i,j=1}^{n}$ with the bandwidth $b_{n}$ where
$a_{l,-n+1}, \cdots,a_{l,n-1}$ are
 complex numbers and $l=1,\ldots,p$, we have the trace formula
 \begin{equation} \label{basic:lem1}
 \mathrm{tr}( T_{1,n}\cdots T_{p,n})=\sum_{i=1}^{n}\,\sum_{J}
a_{J}\,I_{J} \ \large{\delta}_{0,\sum\limits_{l=1}^{p}j_{l}},\quad
\quad p\in \mathbb{N}.
 \end{equation}
  Here
$J=(j_{1},\ldots,j_{p})\in \{-b_{n},\ldots,b_{n}\}^{p}$,
$a_{J}=\prod^{p}_{l=1}a_{l,j_{l}}$,
$I_{J}=\prod^{p}_{k=1}\chi_{[1,n]}(i+\sum_{l=1}^{k}j_{l})$ and the
summation $\sum_{J}$ runs over all possibilities that $J \in
\{-b_{n},\ldots,b_{n}\}^{p}$.
\end{lemma}

\begin{proof}
For the standard basis $\{e_{1}, \cdots, e_{n}\}$ of the Euclidean
space $\mathbb{R}^{n}$, we have
\[T_{p,n}\, e_{i}=\sum_{j=0}^{b_{n}}a_{p,-j}
B^{j}\,e_{i}+ \sum_{j=1}^{b_{n}}a_{p,j} F^{j}\,e_{i}=
\sum_{j=-b_{n}}^{b_{n}}a_{p,j}\, \chi_{[1,n]}(i+j)\, e_{i+j}.\]
Repeating $T_{l,n}$'s effect on the basis, we have
\[T_{1,n}\cdots T_{p,n}\, e_{i}=
\sum_{j_{1},\cdots,j_{p}=-b_{n}}^{b_{n}}
\prod_{l=1}^{p}a_{l,j_{l}}\prod_{k=1}^{p}\chi_{[1,n]}(i+\sum_{l=1}^{k}j_{l})
\,e_{i+\sum\limits_{l=1}^{p}j_{l}}.\
\]
By $\mbox{tr}(T_{1,n}\cdots
T_{p,n})=\sum\limits_{i=1}^{n}e^{t}_{i}\, T_{1,n}\cdots T_{p,n}
\,e_{i}$, we complete the proof.
\end{proof}

%\begin{lemma}\label{toeplitzlemma}
%For a Toeplitz band matrix $T_{n}=(\eta_{ij}\, a_{i-j})_{i,j=1}^{n}$
%with the bandwidth $b_{n}$ where $a_{-n+1}, \cdots,a_{n-1}$ are
% complex numbers, we have the trace formula
% \begin{equation} \label{basic:lem1}
% \mathrm{tr}(T_{n}^{p})=\sum_{i=1}^{n}\,\sum_{J}
%a_{J}\,I_{J} \ \large{\delta}_{0,\sum\limits_{k=1}^{p}j_{k}},\quad
%\quad p\in \mathbb{N}.
% \end{equation}\textbf{}
%  Here
%$J=(j_{1},\ldots,j_{p})\in \{-b_{n},\ldots,b_{n}\}^{p}$,
%$a_{J}=\prod^{p}_{k=1}a_{j_{k}}$,
%$I_{J}=\prod^{p}_{k=1}\chi_{[1,N]}(i+\sum_{l=1}^{k}j_{l})$ and the
%summation $\sum_{J}$ runs over all possibilities that $J \in
%\{-b_{n},\ldots,b_{n}\}^{p}$.
%\end{lemma}
%
%\begin{proof}
%For the standard basis $\{e_{1}, \cdots, e_{n}\}$ of the Euclidean
%space $\mathbb{R}^{n}$, we have
%\[T_{n}\, e_{i}=\sum_{j=0}^{b_{n}}a_{-j}
%B^{j}\,e_{i}+ \sum_{j=1}^{b_{n}}a_{j} F^{j}\,e_{i}=
%\sum_{j=-b_{n}}^{b_{n}}a_{j}\, \chi_{[1,n]}(i+j)\, e_{i+j}.\]
%Repeating $T_{n}$'s effect on the basis, we have
%\[T_{n}^{p}\, e_{i}=
%\sum_{j_{1},\cdots,j_{p}=-b_{n}}^{b_{n}}
%\prod_{k=1}^{p}a_{j_{k}}\prod_{k=1}^{p}\chi_{[1,n]}(i+\sum_{l=1}^{k}j_{l})
%\,e_{i+\sum\limits_{k=1}^{p}j_{k}}.\
%\]
%By $\mbox{tr}(T_{n}^{p})=\sum\limits_{i=1}^{n}e^{t}_{i}\, T_{n}^{p}
%\,e_{i}$, we complete the proof.
%\end{proof}

We will mainly use the above trace formula in the case where
$T_{1,n}=\cdots= T_{p,n}=T_{n}$. Since $a_{0}\equiv 0$, from
Kronecker delta symbol in the trace formula of (\ref{basic:lem1}),
it suffices to consider these $J=(j_{1},\ldots,j_{p})\in \{\pm
1,\ldots,\pm b_{n}\}^{p}$ with the addition of
$\sum\limits_{k=1}^{p}j_{k}=0$. We remark that $a_{0}\equiv 0$ is
not necessary to Theorem \ref{limitmoment}. In fact it is sufficient
to ensure Theorem \ref{limitmoment} if all finite moments of random
variable $a_{0}$ exist and its expectation is zero.

\begin{definition} Let $J=(j_{1},\ldots,j_{p})\in \{\pm 1,\ldots,\pm b_{n}\}^{p}$, we
say $J$ is balanced if $\sum\limits_{k=1}^{p}j_{k}=0$.  The
component $j_{u}$ of $J$ is said to be coincident with $j_{v}$ if
$|j_{u}|=|j_{v}|$ for $1\leq u \neq v \leq p$.
\end{definition}

For  $J\in \{\pm 1,\ldots,\pm b_{n}\}^{p}$, we construct a set of
numbers with multiplicities
\be{\label{projection}S_{J}=\{|j_{1}|,\ldots,|j_{p}|\}}.\ee We call
$S_{J}$ the \emph{projection} of $J$.
\par
The balanced $J$'s can be classified into three categories.

Category 1 (denoted by $\Gamma_{1}(p)$): $J$ is said to belong to
category 1 if each of its components is coincident with exactly one
other component of the opposite sign. It is obvious that
$\Gamma_{1}(p)$ is an empty set when $p$ is odd.

Category 2 ($\Gamma_{2}(p)$) consists of all those vectors such that
$S_{J}$ has at least one number with multiplicity 1.

Category 3 ($\Gamma_{3}(p)$) consists of all other  balanced vectors
in $\{\pm 1,\ldots,\pm b_{n}\}^{p}$. For $J  \in \Gamma_{3}(p)$,
either $S_{J}$ has one number of at least 3 multiplicity, or each of
$S_{J}$ has multiplicity 2 but at least two of the components are
the same, which are denoted respectively by $\Gamma_{31}(p)$ and
$\Gamma_{32}(p)$.

We are now ready to prove Theorem \ref{limitmoment}.

\begin{proof}[Proof of Theorem \ref{limitmoment}]  By Lemma \ref{toeplitzlemma}, we have
\begin{equation}
 \mathbb{E}[\frac{1}{n} \mathrm{tr} (A_{n}^{2k})]=\frac{1}{n b_{n}^{k}}
 \sum_{i=1}^{n}\,\sum_{J}
\mathbb{E}[a_{J}]\,I_{J} \
\large{\delta}_{0,\sum\limits_{l=1}^{2k}j_{l}}=\sum\nolimits_{1}+\sum\nolimits_{2}+\sum\nolimits_{3},
\end{equation}
where\be{\sum\nolimits_{l}=\frac{1}{n b_{n}^{k}}
 \sum_{i=1}^{n}\,\sum_{J\in \Gamma_{l}(2k)}
a_{J}\,I_{J} },\ \ \  l=1,2,3.\ee

By the definition of the categories and the assumptions on the
entries of the random matrices, we obtain
$$\sum\nolimits_{2}=0.$$

Next, we divide $\sum\nolimits_{3}$ into two parts
$$\sum\nolimits_{3}=\sum\nolimits_{31}+\sum\nolimits_{32},$$
where\be{\sum\nolimits_{3l}=\frac{1}{n b_{n}^{k}}
 \sum_{i=1}^{n}\,\sum_{J\in \Gamma_{3l}(2k)}
a_{J}\,I_{J} },\ \ \  l=1,2.\ee For $J\in \Gamma_{3}(2k)$, we denote
the number of distinct elements of $S_{J}$ by $t$. By the definition
of the category, we have $t\leq k$. Note that the random variables
whose subscripts have different absolute values are independent.
Once we have specified the distinct numbers of $S_{J}$, the
subscripts $j_{1},\cdots,j_{2k}$ are determined in at most
$2^{2k}k^{2k}$ ways. If $J\in \Gamma_{31}(2k)$, then $t\leq
\frac{2k-1}{2}$. Again by independence and the assumptions on the
matrix elements
 (\ref{symmetrictoeplitz2}), we find
\[|\sum\nolimits_{31}|\leq \frac{1}{n b_{n}^{k}}
 \sum_{i=1}^{n}O(b_{n}^{\frac{2k-1}{2}})=o(1).\]

When  $J\in \Gamma_{32}(2k)$, there exist $p_{0}, q_{0}\in [2k]$
such that
$$j_{p_{0}}=j_{q_{0}}=\frac{1}{2}(j_{p_{0}}+j_{q_{0}}-\sum\limits_{q=1}^{2k}j_{q}).$$
We can choose the other at most $k-1$  distinct numbers, which
determine $j_{p_{0}}=j_{q_{0}}$. This shows that there is a loss of
at least
 one degree of freedom, thus the contribution of such terms is $O(n^{-1})$, i.e.
\[|\sum\nolimits_{32}|=o(1).\]

Since the main contribution comes from the category 1, each term
$\mathbb{E}[a_{J}]=1$ for $J\in \Gamma_{1}(2k)$.
   So we can rewrite \begin{equation}
\label{integral:sum}
\mathbb{E}[\frac{1}{n}\mathrm{tr}(A_{n}^{2k})]=o(1)+ \frac{1}{n
b_{n}^{k}}\sum_{\pi \in \mathcal{P}_{2}(2k)
}\sum_{i=1}^{n}\,\sum_{j_{1},\cdots,j_{k}=-b_{n}}^{b_{n}}
\prod_{l=1}^{2k}\chi_{[1,n]}(i+\sum_{q=1}^{l}
\epsilon_{\pi}(q)j_{\pi(q)}). \end{equation}

For fixed $\pi \in \mathcal{P}_{2}(2k)$,
$$
\frac{1}{n
b_{n}^{k}}\sum_{i=1}^{n}\,\sum_{j_{1},\cdots,j_{k}=-b_{n}}^{b_{n}}
\prod_{l=1}^{2k}\chi_{[1,n]}(i+\sum_{q=1}^{l}
\epsilon_{\pi}(q)j_{\pi(q)}),$$ i.e.
$$
\frac{1}{n
b_{n}^{k}}\sum_{i=1}^{n}\,\sum_{j_{1},\cdots,j_{k}=-b_{n}}^{b_{n}}
\prod_{l=1}^{2k}\chi_{[\frac{1}{n},1]}(\frac{i}{n}+\sum_{q=1}^{l}
\epsilon_{\pi}(q)\frac{b_{n}}{n}\frac{j_{\pi(q)}}{b_{n}})$$
 can be considered as a
Riemann sum of the definite integral
$$\int_{[0,1]\times [-1,1]^{k}}\prod_{l=1}^{2k}I_{[0,1]}(x_{0}+b\sum_{q=1}^{l}
\epsilon_{\pi}(q)x_{\pi(q)}) \prod_{l=0}^{k}\mathrm{d}\, x_{l}.$$

As in the above arguments, by Lemma \ref{toeplitzlemma} and the
assumptions on the matrix elements
 (\ref{symmetrictoeplitz2}), we have \be{\left|\mathbb{E}[\frac{1}{n}\mathrm{tr}(A_{n}^{2k+1})]\right|\leq
 \frac{1}{n b_{n}^{\frac{2k+1}{2}}}
 \sum_{i=1}^{n}\,O(b_{n}^{k})=o(1)}\ee since
$\mathcal{P}_{2}(2k-1)=\emptyset$.

This completes the proof.
\end{proof}

%\begin{remark}
%It is also worth to notice that
%$\mathbb{E}[\frac{1}{n}\mathrm{tr}(A_{n}^{2k-1})]=o(1)$ since
%$\mathcal{P}_{2}(2k-1)=\emptyset$ though we don't use it directly.
%\end{remark}

\begin{remark}When $b=0$, we can easily get $M_{2k}=2^{k}(2k-1)!!$. This
is just the $2k$- moment of the  normal distribution with variance
2, which is also obtained independently by   Basak and  Bose
\cite{bb} and  Kargin \cite{kargin}. However, for $b>0$ it is quite
difficult to calculate $M_{2k}$ because the integrals in the sum of
(\ref{bandtoeplitz:moment}) are not all the same for different
partitions $\pi$'s, some of which are too hard to evaluate.
\end{remark}

\section{covariance}\label{variance}
\setcounter{equation}{0} In this section we evaluate the covariance
of $\omega_{p}$ and $\omega_{q}$. Recall
\be{\omega_{p}=\frac{\sqrt{b_{n}}}{n}\left(\mathrm{tr}(A_{n}^{p})-\mathbb{E}[\mathrm{tr}(A_{n}^{p})]\right)}
=\frac{1}{n b_{n}^{\frac{p-1}{2}}}
 \sum_{i=1}^{n}\,\sum_{J}I_{J}\left(a_{J}-
\mathbb{E}[a_{J}]\right) \
\large{\delta}_{0,\sum\limits_{l=1}^{p}j_{l}},\ee thus
\begin{align}
 \mathbb{E}[\omega_{p}\,\omega_{q}]&=\frac{1}{n^{2} b_{n}^{\frac{p+q}{2}-1}}
 \sum_{i,i'}\,\sum_{J,J'}I_{J}\,I_{J'}
\mathbb{E}[\left(a_{J}- \mathbb{E}[a_{J}]\right)\left(a_{J'}-
\mathbb{E}[a_{J'}]\right)]\nonumber\\
&=\frac{1}{n^{2} b_{n}^{\frac{p+q}{2}-1}}
 \sum_{i,i'}\,\sum_{J,J'}I_{J}\,I_{J'}
\left(\mathbb{E}[a_{J}a_{J'}]-
\mathbb{E}[a_{J}]\mathbb{E}[a_{J'}]\right),\label{covariance}
\end{align}
where the summation $\sum\limits_{J,J'}$ runs over all % possibilities
%that
 balanced vectors  $J\in \{\pm 1,\ldots,\pm b_{n}\}^{p}$
and $J'\in \{\pm 1,\ldots,\pm b_{n}\}^{q}$. %Obviously, since the
%odd moments vanish, when $p+q$ is odd $
%\mathbb{E}[\omega_{p}\,\omega_{q}]=0$.
%We will consider the case
%that $p+q$ is even in this section.

The main result of this section can be stated as follows:
\begin{theorem}\label{covariancecalculation}Using the notations in section
\ref{Integrals associated with pair partitions}, for fixed $p,\
q\geq 2$,  as $n\longrightarrow \infty$ we have
\begin{align}
 \mathbb{E}[\omega_{p}\,\omega_{q}]\longrightarrow \sigma_{p,q}=
\sum_{\pi\in \mathcal{P}_{2
}(p,q)}\left(f^{-}_{I}(\pi)+f^{+}_{I}(\pi)\right)+(\kappa-1)\sum_{\pi\in
\mathcal{P}_{2,4}(p,q)}\left(f^{-}_{II}(\pi)+f^{+}_{II}(\pi)\right)
\end{align}
when $p+q$ is even and
\begin{align}
 \mathbb{E}[\omega_{p}\,\omega_{q}]=o(1)
\end{align}
when $p+q$ is odd.
\end{theorem}

When $p=q$ we denote $\sigma_{p,q}$ by $\sigma^{2}_{p}$, where
$\sigma_{p}$ denotes the standard deviation. From the above theorem,
we can obtain the variance of $\omega_{Q}$ in Theorem \ref{tpg}
\be{\sigma^{2}_{Q}=\sum_{i=2}^{p}\sum_{j=2}^{p}q_{i}q_{j}
\sigma_{i,j}}.\ee

By the independence of matrix entries, the only non-zero terms in
the sum of (\ref{covariance}) come from pairs of balanced vectors
$J=(j_{1},\ldots,j_{p})$ and $J'=(j'_{1},\ldots,j'_{q})$ such that

(i) The \emph{projections}  $S_{J}$ and $S_{J'}$ of $J$ and $J'$
have at least one element in common;

(ii) Each number in the union of $S_{J}$ and $S_{J'}$ occurs at
least two times.

\begin{definition}Any pair of balanced vectors
$J=(j_{1},\ldots,j_{p})$ and $J'=(j_{1}',\ldots,j_{q}')$ satisfying
(i) is called correlated. If $j_u \in J$($j_u$ is a component of
$J$) and $|j_u| \in S_J \bigcap S_{J'}$, then $j_u$ is called a
joint point of the ordered correlated pair.
\end{definition}

To observe which correlated pairs lead to the main contribution to
the covariance, we next construct a balanced vector of dimension
($p+q-2$) from each correlated pair $J$ of dimension $p$ and $J'$ of
dimension $q$. Although the corresponding map of correlated pairs to
such balanced vectors is not one to one, the number of pre-images
for a balanced vector of dimension ($p+q-2$) is finite (only
depending on $p$ and $q$). We will study the resulting balanced
vectors in a similar way in section \ref{expectation}.

\begin{proof}[Proof of Theorem \ref{covariancecalculation}]
Let us first  construct a map from the ordered correlated pair $J$
and $J'$ as follows. Let $j_{u}\in J$ be the first joint point(whose
subscript is the smallest) of the ordered correlated pair $J$ and $J'$,
and let $j'_v$ be the first element in $J'$ such that
$|j_{u}|=|j'_{v}|$. If $j_{u}=-j'_{v}$, we construct a vector
$L=(l_{1},\ldots,l_{p+q-2})$ such that
\begin{align}l_{1}=j_{1},\ldots,&l_{u-1}=j_{u-1},
l_{u}=j'_{1},\ldots,l_{u+v-2}=j'_{v-1},\nonumber\\
&l_{u+v-1}=j'_{v+1}, \ldots,l_{u+q-2}=j'_{q},
l_{u+q-1}=j_{u+1},\ldots l_{p+q-2}=j_{p}.\nonumber
\end{align}
It is obvious that $$\sum\limits_{k=1}^{p+q-2}l_{k}=0,$$ so $L$ is
balanced. If $j_{u}=j'_{v}$, then from $J$ and
$-J'=(-j'_{1},\ldots,-j'_{q})$ we proceed as in the  way above. We
call this process of constructing $L$ from $J$ and $J'$ a
$reduction$ step and denote it by $L=J{\bigvee}_{|j_u|}J'$.
\begin{remark}
From the construction above, for any joint point of $J$ and $J'$ a
reduction step can be done in the same way. Given $\theta \in
S_J\bigcap S_{J'}$, when saying $J{\bigvee}_{\theta}J'$, we mean
that there exists certain joint point of $J$ and $J'$ $j_u$
satisfying $|j_u|=\theta$ and $J{\bigvee}_{\theta}J'$ is the vector
after this reduction step. In this section, $j_u$ is always the
first joint point. While in section \ref{highermoments}, $j_u$ may
denote other joint points which is clear in the context.
\end{remark}
\par
Notice that the reduction might cause the appearance of one number
with multiplicity 1 in $S_{L}$, although each number in the union of
$S_{J}$ and $S_{J'}$ occurs at least two times. If so, the resulting
 number with multiplicity 1 in $S_{L}$ must be coincident with the joint point
 $j_{u}$. In addition, to estimate which terms lead to main contribution to higher moments of $\mathrm{tr}(A_n^p)$, we will use the $reduction$ steps and mark the appearance of the numbers with
multiplicity 1 in section \ref{highermoments}.

 Next, we assume we have a balanced vector $L$ of dimension ($p+q-2$), we shall estimate in how many different
 ways  it can be obtained from correlated pairs of dimensions $p$ and $q$. First, we have to choose
 some component
 $l_{u}$ in the first half of the vector, $1\leq u \leq p$ such that
 \be{\label{premagecondition1}\left| \sum\limits_{i=u}^{u+q-2}l_{i}\right|\neq |l_{j}|,\ \ j=1,\ldots,u-1.}\ee
 Set $J=(j_{1},\ldots,j_{p})$ with
\be{j_{1}=l_{1},\ldots,j_{u-1}=l_{u-1},j_{u}=\sum\limits_{i=u}^{u+q-2}l_{i},j_{u+1}=l_{u+q-1},
\ldots,j_{p}=l_{p+q-2}. }\ee  We also  have to choose
 some component
 $l_{u+v-1}$, $1\leq v \leq q-1$ such that
 \be{\label{premagecondition2}\left| \sum\limits_{i=u}^{u+q-2}l_{i}\right|\neq |l_{j}|,\ \ j=u,\ldots,u+v-2}\ee whenever $v\geq 2$.
  Set $J'=(j'_{1},\ldots,j'_{q})$ with
\be{j'_{1}=l_{u},\ldots,j'_{v-1}=l_{u+v-2},j'_{v}=-\sum\limits_{i=u}^{u+q-2}l_{i},j'_{v+1}=l_{u+v-1},
\ldots,j'_{p}=l_{u+q-2}.}\ee If $j_{u}$ is the joint point of the
constructed correlated pair $J$ and $J'$ and $j'_v$ is the
corresponding element in $J'$, then the pair $\{J,J'\}$ or
$\{J,-J'\}$ is the pre-image of $L$. Note that since when $u=v=1$
the conditions (\ref{premagecondition1}) and
(\ref{premagecondition2}) are satisfied, the pre-image of $L$ always
exists. A simple estimation shows that the number of pre-images of
$L$ is at most $2pq$, not depending on $n$ (we will see this fact
plays an important role in the estimation of higher moments in
section \ref{highermoments}).

Since there is at most one element with  multiplicity 1 in $S_{L}$,
if there is one number with multiplicity 1, then the number will be
determined by others because of the balance of $L$. Consequently,
the degree of freedom for such terms is at most $\frac{p+q-2-1}{2}$.
Therefore, the sum of these terms will be $O(b^{-1/2}_{n})$, which
can be omitted. Now we suppose each number in $S_{L}$ occurs at
least two times. Recall the procedure in section \ref{expectation},
and we know that the main contribution to the covariance
(\ref{covariance}) comes from the $L\in \Gamma_{1}(p+q-2)$, which
implies $
 \mathbb{E}[\omega_{p}\,\omega_{q}]=o(1)$
when $p+q$ is odd. When $p+q$ is even, for $L\in \Gamma_{1}(p+q-2)$
the weight \be{\mathbb{E}[a_{J}a_{J^{'}}]-
\mathbb{E}[a_{J}]\mathbb{E}[a_{J^{'}}]=\mathbb{E}[\prod^{p}_{s=1}a_{j_{s}}\prod^{q}_{t=1}a_{j'_{t}}]-
\mathbb{E}[\prod^{p}_{s=1}a_{j_{s}}]\mathbb{E}[\prod^{q}_{t=1}a_{j'_{t}}]
}\ee equals to 1 if $j_{u}$ is not coincident with any component of
$L$; otherwise the weight is either
$\mathbb{E}[|a_{j_{u}}|^{4}]=\kappa$ or
$\mathbb{E}[|a_{j_{u}}|^{4}]-\left(\mathbb{E}[|a_{j_{u}}|^{2}]\right)^{2}=\kappa-1$.

So far  we have found such terms leading to the main contribution,
now we calculate the variance. Based on whether or not the fourth
moment appears, we evaluate the covariance. If the fourth moment
doesn't appear, then $j_{1},\ldots,j_{p}, j'_{1},\ldots,j'_{q}$
match in pairs. In the abstract, by their subscripts  they can be
treated as pair partitions of $\{1, 2, \ldots, p, p+1, \ldots,
p+q\}$ but with at least one crossing match (i.e.,
$\mathcal{P}_{2}(p,q)$ as in section \ref{Integrals associated with
pair partitions}). Thus, for every $\pi \in \mathcal{P}_{2}(p,q)$,
the summation can be a Riemann sum and its limit becomes
$f^{-}_{I}(\pi)$ ( it is $f^{+}_{I}(\pi)$
 when the first
 coincident components in $J$ and $J'$   have the  same
 sign).  On the other hand,   if the fourth moment does  appear,
 then
 $j_{1},\ldots,j_{p}, j'_{1},\ldots,j'_{q}$ match in pairs except that
 there exist a block with four elements. Therefore, from the balance of $\sum\limits_{k=1}^{p}j_{k}=0$
 and $\sum\limits_{k=1}^{p}j'_{k}=0$, we know that the main contribution must come from such partitions:
$j_{1},\ldots,j_{p}$ and  $j'_{1},\ldots,j'_{q}$ both form pair
partitions; the block with four elements take respectively from a
pair of $j_{1},\ldots,j_{p}$ and  $j'_{1},\ldots,j'_{q}$. Otherwise,
the degree of freedom decreases by at least one. Similarly, for
every $\pi \in \mathcal{P}_{2,4}(p,q)$, the corresponding summation
can be a Riemann sum and its limit becomes $f^{-}_{II}(\pi)$ (it is
$f^{+}_{II}(\pi)$
 when the first
 coincident components in $J$ and $J'$   have the  same
 sign).

 % is not
%coincident with any component of $L$,

In a similar way as in section \ref{expectation}, noting that the
 coincident components in $J$ and $J'$ may have the same or opposite
 sign, we conclude
   with %from remark \ref{typeIIremark} and
   the
notations in section \ref{Integrals associated with pair partitions}
that
\begin{align}
 \mathbb{E}[\omega_{p}\,\omega_{q}]\longrightarrow
\sum_{\pi\in \mathcal{P}_{2
}(p,q)}\left(f^{-}_{I}(\pi)+f^{+}_{I}(\pi)\right)+(\kappa-1)\sum_{\pi\in
\mathcal{P}_{2,4}(p,q)}\left(f^{-}_{II}(\pi)+f^{+}_{II}(\pi)\right)
\end{align}
as $n\longrightarrow \infty$.

This completes the proof.
\end{proof}

\section{Higher Moments of $\mathrm{tr}(A_n^p)$}\label{highermoments}
\setcounter{equation}{0} Let $\mathfrak{B}_{n,p}$ denote the set of
all balanced vectors $J=(j_{1},\ldots,j_{p})\in \{\pm 1,\ldots,\pm
b_{n}\}^{p}$. Let $\mathfrak{B}_{n,p,i}$ ($1\leq i \leq n$) be a
subset of $\mathfrak{B}_{n,p}$ such that $J \in
\mathfrak{B}_{n,p,i}$ if and only if
\[\forall\,t \in \{1,\ldots,p\},\quad 1\leq i+\sum_{q=1}^{t}j_{q} \leq
n.
\]
With these notations, Lemma \ref{toeplitzlemma} could be rewritten
as \be \label{trformula}
\mathrm{tr}(T_n^{p})=\sum_{i=1}^{n}\,\sum_{J \in
\mathfrak{B}_{n,p,i}} a_{J}.\ee
\\
\par
To finish the proof of Theorem \ref{tpg}, it is sufficient to show that  given $ p_{1}, p_{2}, \cdots, p_{l}\geq 2$ and $l\geq 1$, as $n \rightarrow \infty$ we have
\begin{equation}\label{jointmoments}\mathbb{E}[\omega_{p_{1}}\omega_{p_{2}}\cdots \omega_{p_{l}}]\longrightarrow \mathbb{E}[g_{p_{1}}g_{p_{2}}\cdots g_{p_{l}}],
\end{equation}
where $\{g_{p}\}_{p \geq 2}$ is a Gaussian family with covariances $\sigma_{p,q}=\mathbb{E}[{g}_{p}g_{q}]$.

%%\be\label{main term}\mathbb{E}[\omega^{2k}_{p}]=
%% \big(\frac{\sqrt{b_n}}{n}\big)^{2k}\mathbb{E}[(\mathrm{tr} (A_n^p)-
%%\mathbb{E}[\mathrm{tr}(A_n^p)])^{2k}]=(2k-1)!!\cdot(\sigma_p^{2k}+o(1))
%%\end{equation}
%%and
%%\begin{equation}\label{bit term}\mathbb{E}[\omega^{2k+1}_{p}]=
%%\big(\frac{\sqrt{b_n}}{n}\big)^{2k+1}\mathbb{E}[(\mathrm{tr}
%%(A_n^p)- \mathbb{E}[\mathrm{tr}(A_n^p)])^{2k+1}]=o(1).
%%\end{equation}
Then a CLT for
\be{\omega_{Q}=\frac{\sqrt{b_{n}}}{n}\left(\mathrm{tr}Q(A_{n})-\mathbb{E}[\mathrm{tr}Q(A_{n})]\right)}\ee
follows, with the variance
\be{\sigma^{2}_{Q}=\sum_{i=2}^{p}\sum_{j=2}^{p}q_{i}q_{j}
\sigma_{i,j}}.\ee

 The main idea is
rather straightforward: in an  analogous way to the one used in Eq.
(\ref{covariance}), we will deal with
%the basic formula
\begin{align}\label{fluctuation
formula}\mathbb{E}&[\omega^{k_{1}}_{p_{1}}\cdots \omega^{k_{l}}_{p_{l}}]=\nonumber\\
&n^{-l} \cdot
b_n^{-\frac{p_1+\cdots+p_l-l}{2}}
\sum_{i_1,\dots,i_l=1}^{n}\,\sum_{J_1 \in
\mathfrak{B}_{n,p_1,i_1},\dots,J_l \in
\mathfrak{B}_{n,p_l,i_l}}\prod_{t=1}^{l}I_{J_{t}}\mathbb{E}\left[
\prod_{t=1}^{l} (a_{J_t}-\mathbb{E}[a_{J_t}])\right].\end{align}

Remember that two balanced vectors $J=(j_{1},\ldots,j_{p})$ and
$J'=(j_{1}',\ldots,j_{q}')$ are called correlated if the
corresponding  \emph{projections} $S_{J}$ and $S_{J'}$ of $J$ and
$J'$ have at least one element in common.

\begin{definition}
Given a set of balanced vectors $\{J_1,J_2,\ldots,J_l\}$, a subset
of balanced vectors $J_{m_{j_1}}$,$J_{m_{j_2}}$,\ldots,$J_{m_{j_t}}$
is called a cluster if \\
1)\quad for any pair $J_{m_i}$,$J_{m_j}$ from the subset one can
find a chain of vectors $J_{m_s}$, also belongs to the subsets,
which starts with $J_{m_{i}}$ ends with $J_{m_j}$, such that any two
neighboring vectors are correlated; \\
2)\quad the subset $J_{m_{j_1}}$,$J_{m_{j_2}}$,\ldots,$J_{m_{j_t}}$
cannot be enlarged with the preservation of 1).
\end{definition}
It is clear that the vectors corresponding to different clusters are
disjoint. By this reason the mathematical expectation in
(\ref{fluctuation formula}) decomposes into the product of
mathematical expectations corresponding to different clusters. We
shall show that the leading contribution to (\ref{fluctuation
formula}) comes from products where all clusters consist exactly of
two vectors, as is stated in Lemma \ref{cluster} below. %%Based on
%%this result, formulae (\ref{main term}) and ( \ref{bit term})
%%essentially follow from the case of variance considered in section
%%\ref{variance}.

\begin{lemma}\label{cluster}
Provided $l\geq 3$, we have  \be\label{eq:cluster}
n^{-l} \cdot
b_n^{-\frac{p_1+\cdots+p_l-l}{2}}
\sum_{i_1,\dots,i_l=1}^{n}\,\sum_{J_1 \in
\mathfrak{B}_{n,p_1,i_1},\dots,J_l \in
\mathfrak{B}_{n,p_l,i_l}}\prod_{t=1}^{l}I_{J_{t}} \mathbb{E}\left[
{\prod_{t=1}^{l}}^{\star} (a_{J_t}-\mathbb{E}[a_{J_t}])\right]=o(1)
\ee where the product ${\prod}^{\star}$ in (\ref{eq:cluster}) is
taken
over $l$ vectors which exactly form a $cluster$. %of length $l$.
\end{lemma}

%\begin{proof}[Proof of Lemma \ref{cluster}]
For fixed $p$, all the involved moments no higher than $p$ are
$O(1)$ because of uniform boundedness of matrix entries. On the
other hand, $0\leq \prod_{t=1}^{l}I_{J_{t}}\leq 1$. So to prove
Lemma \ref{cluster}, we just need to count the number of terms in
(\ref{eq:cluster}).
 As before, to complete the estimation it suffices to replace
$\Pnki$ by $\Pnk$. That is, it suffices to prove
\be\label{eq:cluster1} b_n^{-\frac{p_1+\cdots+p_l-l}{2}}
{\sum_{J_1  \in \mathfrak{B}_{n,p_{1}},\ldots,J_l \in \mathfrak{B}_{n,p_{l}}}}^{\star}\  1=o(1)  \ee
where the summation ${\sum}^{\star}$ is taken over $l$ vectors which
exactly form a $cluster$.

Instead of Lemma \ref{cluster}, we will prove

\begin{lemma}\label{general case}
Provided $l \geq 3$, $\mathbf{p}=(p_1,\ldots,p_l)$ with positive
integers $p_1,\ldots,p_l\geq 2$. Let $\mathfrak{B}_{\mathbf{p}}$ be
a subset of the Cartesian product $\mathfrak{B}_{n,p_1} \times \dots
\times\mathfrak{B}_{n,p_l}$ such that $(J_1,J_2,\ldots,J_l)\in
\mathfrak{B}_{\mathbf{p}}$
if and only if \\
%(i)$0 \overline{\in} \bigcup_{i=1}^{l}J_i$\\
(i) any element in $\bigcup_{i=1}^{l} S_{J_i}$ has at least two multiplicities in the union;\\
(ii) $J_1,J_2,\ldots, J_l$ make a \emph{cluster};\\
(iii) $0 \overline{\in} \bigcup_{i=1}^{l}S_{J_i}$.\\
Then we claim that
\be{\textrm{card}|\mathfrak{B}_{\mathbf{p}}|=o({b_n}^{\frac{p_1+p_2+\cdots+p_l-l}{2}}).}\ee
\end{lemma}

 Notice that we list the condition (iii) which
looks redundant from $J\in \{\pm 1,\ldots,\pm b_{n}\}^{p}$ to
emphasize the importance of $0$ (i.e., the diagonal matrix entry
$a_{0}$). In fact, if $p_1, p_2, \cdots, p_l$ are all even, even for
these $J\in \{0, \pm 1,\ldots,\pm b_{n}\}^{p}$ but under the
conditions (i) and (ii), we can still get the above estimation.

\begin{proof}[Proof of Lemma \ref{general case}]
The intuitive idea of the proof is as follows: from condition (i) in
Lemma \ref{general case}, regardless of the correlating condition,
the cardinality of $\mathfrak{B}_{\mathbf{p}}$(denoted by
card$|\mathfrak{B}_{\mathbf{p}}|$ for short) is
$O(b_{n}^{\frac{p_1+p_2+\cdots+p_l}{2}})$. We could say that the
freedom degree is $\frac{p_1+p_2+\cdots+p_l}{2}$. But each
correlation means two vectors share a common element so that it will
decrease the freedom degree by one. To form a \emph{cluster} we need
$l-1$ correlations. So
$\textrm{card}|\mathfrak{B}_{\mathbf{p}}|=O(b_{n}^{\frac{p_1+p_2+\cdots+p_l}{2}-(l-1)})$.
When $l \geq 3$, $l-1 > \frac{l}{2}$. Thus we obtain the desired
estimation in Lemma \ref{general case}. However,  unfortunately
adding one correlation does not necessarily lead the freedom degree
to decrease by one. Some may be redundant. So we have to make use of
 the correlations more efficiently.

 As in section \ref{variance}, we do the \emph{reduction} steps
as long as the structure of the $cluster$ is preserved. To say
precisely, we start from checking $J_1$ and $J_2$. If $j_u$ is the
first joint point of $J_1$ and $J_2$ satisfying the condition that
$J_{1}{\bigvee}_{|j_u|}J_{2}$ can still form a \emph{cluster} with
the other vectors, then we do this $reduction$ step. If this kind of
$j_u$ does not exist, we turn to check $J_1$ and $J_3$ in the same
way. After each $reduction$ step, we have a new cluster of vectors
$\widetilde{J_{1}},\widetilde{J_{2}},\ldots,\widetilde{J_{\widetilde{l}}}$.
We continue to check $\widetilde{J_{1}}$ and $\widetilde{J_{2}}$ as
before. If we cannot do any $reduction$ step, we stop.
\par
Suppose that we did $m$ \emph{reduction} steps in total. Then we
have a new $cluster$ of vectors $J'_1,J'_2,\ldots,J'_{l'}$ and the
dimension of $J'_i$ is $p'_i$ for $1 \leq i \leq l'$. From the
reduction process, for any $\theta \in S_{J'_{{\alpha}_1}}\bigcap
S_{J'_{{\alpha}_2}}$,
$J'_{{\alpha}_1}{\bigvee}_{\theta}J'_{{\alpha}_2}$ cannot form a
\emph{cluster} with the other vectors. The resulting $cluster$ still
satisfies condition (ii) in Lemma \ref{general case}. However, the
condition (i) may fail because the joint point of a pair of
correlated vectors could be triple multiplicity, thus after a
\emph{reduction} step its multiplicity becomes one.
\par
Note that after a $reduction$ step  the number of pre-images of the
resulting vector  only depends on the dimensions of the involving
vectors, not depending on $n$. Thus we  only need to estimate the
degree of freedom of the reduced vectors $J'_1,J'_2,\ldots,J'_{l'}$.
%from the the procession of $reduction$ step.
\par
Since after one \emph{reduction} step the total dimension of vectors
will decrease by two and the number of vectors will decrease by one,
we have
\be\label{relation1} \sum_{i=1}^{l}p_i=\sum_{i=1}^{l'}p'_i+2m \ee %
and  \be\label{relation2}l=l'+m. \ee Denote by $l_0$ the number of
single multiplicity elements in $\bigcup_{i=1}^{l'}S_{J'_i}$. Since
one \rs \  will add at most one element with single multiplicity,
therefore \be\label{inequ:l_0} l_0\leq m. \ee

%-----------------------------------------------------------------wait for rewrite----------------------------------------------------------------------- To complete the proof we need some definitions and notions.
Below, we will proceed according to two cases: $l'>1$ and $l'=1$.

\begin{flushleft}
\Large{$\mathbf{In\: the\: case\:} l'>1$}
\end{flushleft}
To complete the proof of this case, we need some definitions and
notations.
\par
Let $\mathcal{U}$ be the set consisting of all elements which
belongs to at least two of $S_{J'_{1}}, \ldots,  S_{J'_{l'}}$, i.e.
$\mathcal{U}=\{\theta|\exists i\neq j \quad \mathrm{s.t.}\quad
\theta \in S_{J'_i}\bigcap S_{J'_j} \}$. Since $l'>1$ and $J'_{1},
\ldots, J'_{l'}$ forms a cluster, we get $\mathcal{U}\neq \emptyset$
 . For any $\theta\in\bigcup_{i=1}^{l'}S_{J'_i}$, set
$H_{\theta}=:\{J'_{i}|\,\theta \  \mathrm{or} -\theta \in
J'_{i},\quad 1\leq i\leq l'\}$ and denote the number of vectors in
$H_{\theta}$ by $h_{\theta}=\mathrm{card}|H_{\theta}|$. Obviously,
$\{J'_{i}|i\leq l'\}=\bigcup_{\theta \in \mathcal{U}}H_{\theta}$.

We notice the following three  facts.

Fact 1: For any $\theta\in \mathcal{U}$, $h_{\theta} \geq 3$. In
fact, from the definition of $\ \mathcal{U}$, we know $h_{\theta}
\geq 2$. If $h_{\theta}=2$, the two vectors in $H_{\theta}$ could
still be reduced to one vector and the reduction doesn't affect
their connection with the other vectors, which is a contradiction
with our assumption that $J'_1,J'_2,\ldots,J'_{l'}$ cannot be
reduced.
\par
Fact 2: For any $\theta\in \mathcal{U}$ and $J'_i\in H_{\theta}$,
the multiplicity of $\theta$ in $S_{J'_i}$ is one . Otherwise, there
are two vectors belonging to $H_{\theta}$, for example, $J'_{1},
J'_{2}\in H_{\theta}$ but $S_{J'_{1}}$ has two $\theta$'s, then
$J'_1{\bigvee}_{\theta}J'_2$ could    be a $reduction$ step and
$J'_1{\bigvee}_{\theta}J'_2, J'_{3}, \ldots,  J'_{l'}$ still forms a
cluster.

\par
Fact 3: For any different elements $\theta$  and $\gamma$ in
$\bigcup_{i=1}^{l'}S_{J'_i}$,  card$|H_{\theta}\bigcap
H_{\gamma}|\leq 1$. Otherwise, suppose $J'_i$ and $J'_j$ belongs to
$H_{\theta}\bigcap H_{\gamma}$ and  $J'_k$ is an element of
$H_{\theta}$ other than $J'_i$ and $J'_j$. From Fact 1,  $J'_k$ must
exist. Now $J'_i{\bigvee}_{\theta}J'_k$ can form a $reduction$ step
since $J'_i{\bigvee}_{\theta}J'_k$ could be correlated with $J'_j$
by $\gamma$ and other correlations won't be broken.
\begin{definition}
$\mathcal{V}\subset \mathcal{U}$ is called a \ds of $\{J'_{1},
\ldots, J'_{l'}\}$ if $\{J'_{i}|1\leq i\leq l'\}=\bigcup_{\theta \in
\mathcal{V}}H_{\theta}$.
\end{definition}

Choose  a minimal \ds denoted by $\mathcal{U}_0$,  which means that
any proper subset of $\mathcal{U}_0$ is not a \emph{dominating set}.
Since $\mathcal{U}$ is a finite set, $\mathcal{U}_0$ must exist and
write $\mathcal{U}_0=\{\theta_i|1\leq i\leq t\}$. For any $1\leq
j\leq t$, since $ \mathcal{U}_0\backslash\{\theta_{j}\}$ is not a
\ds,    there exists $J'_{k_j}\in H_{\theta_j}\backslash
\bigcup_{i\neq j }H_{\theta_i}$. Once we have already known the
elements of $\bigcup_{i=1}^{l'}S_{J'_i}\backslash \mathcal{U}_0$,
 we know all the elements in
$S_{J'_{k_j}}$ other than $\theta_j$, thus  $\theta_j$ will be
determined by the balance of $J'_{k_j}$.

Set
$$h=\sum_{i=1}^{t}h_{\theta_i},$$ then the different way of choice
of $\bigcup_{i=1}^{l'}S_{J'_i}\backslash \mathcal{U}_0$ is
$O(b_n^{\frac{\sum_{i=1}^{l'}p'_i-l_0-h}{2}+l_0})$. From Eqs.
(\ref{relation1}) and (\ref{relation2}), we have
\be\label{equ:mainestimating}
\frac{\sum_{i=1}^{l'}p'_i-l_0-h}{2}+l_0=\frac{(\sum_{i=1}^{l}p_i-l)-(m-l_0)-(h-l')}{2}.\ee
Note that $m-l_{0}\geq 0$ from Eq. (\ref{inequ:l_0}).
 If $h>l'$, we have \be
O(b_n^{\frac{(\sum_{i=1}^{l}p_i-l)-(m-l_0)-(h-l')}{2}})=o(b_n^{\frac{\sum_{i=1}^{l}p_i-l}{2}}).
\ee %
If $h=l'$ and $t=1$, the analysis is easy but a little complex. We
will deal with it later.
\par
Now we focus on the situation that $h=l'$ and $t>1$. In this case,
since $\{J'_{i}|i\leq l'\}=\bigcup_{\theta \in
\mathcal{U}_0}H_{\theta}$, $l'\leq \sum_{i=1}^{l'}
\mathrm{card}|H_{\theta_i}|=h$ and the equity is true if\mbox{}f
$H_{\theta_i}\bigcap H_{\theta_j}=\emptyset$ for any $1\leq i\neq j
\leq t$. Because $J'_1,J'_2,\ldots, J'_{l'}$ form a cluster, without
lost of generality, we could assume that $(\bigcup_{J'\in
H_{\theta_1}}S_{J'})\bigcap (\bigcup_{J'\in H_{\theta_2}}S_{J'})\neq
\emptyset$. Thus  there exist $J'_{s_1}\in H_{\theta_1}$
,$J'_{s_2}\in H_{\theta_{2}}$ and $\gamma \in S_{J'_{s_1}}\bigcap
S_{J'_{s_2}}$. We know that $\gamma \in \mathcal{U} \backslash
\mathcal{U}_0$ since $H_{\theta_i}\bigcap H_{\theta_j}=\emptyset$
for any $1\leq i\neq j \leq t$. From Fact 3,
$\mathrm{card}|H_{{\theta}_i}\bigcap H_{\gamma}|=0$ or $1$. Given
the elements of $\bigcup_{i=1}^{l'}S_{J'_i}\backslash
(\mathcal{U}_0\bigcup\{\gamma\})$, $\theta_i$ can be decided by the
the balance of some $J'\in H_{\theta_i}\backslash H_{\gamma}$. Then
$\gamma$ can be decided by any vector in $H_{\gamma}$. Thus we have
$O(b_n^{\frac{\sum_{i=1}^{l'}p'_i-l_0-h-h_{\gamma}}{2}+l_0})$ ways
to decide $J'_1,J'_2,\ldots, J'_{l'}$. From
Eq.(\ref{equ:mainestimating}), one gets
 \be O(b_n^{\frac{\sum_{i=1}^{l'}p'_i-l_0-h-h_{\gamma}}{2}+l_0})=O(b_n^{\frac{(\sum_{i=1}^{l}p_i-l)-(m-l_0)-(h-l')-h_{\gamma}}{2}})
 =o(b_n^{\frac{\sum_{i=1}^{l}p_i-l}{2}}). \ee
\\

%Now we show $h\geq l'$, which is the following lemma.
%\begin{lemma}\label{lemma:h=l'}
%$h\geq l'$ and the equity is true if\mbox{}f $t=1$.
%\end{lemma}
%\begin{proof}[Proof of Lemma \ref{lemma:h=l'}]
%Since $\{J'_{i}|i\leq l'\}=\bigcup_{\theta \in
%\mathcal{V}}H_{\theta}$, we get $l'\leq \sum_{i=1}^{l'}
%\mathrm{card}|H_{\theta_i}|=h$ and the equity is true if\mbox{}f
%$H_{\theta_i}\bigcap H_{\theta_j}=\emptyset$ for any $1\leq i\neq j
%\leq t$. We claim that the equality is not true when $t>1$.
%Otherwise, we suppose that when $t>1$ the equality is still true.
%Without lost of generality, we could assume that $(\bigcup_{J'\in
%H_{\theta_1}}S_{J'})\bigcap (\bigcup_{J'\in H_{\theta_2}}S_{J'})\neq
%\emptyset$ since the vectors form a cluster. Thus  there exist
%$J'_{s_1}\in H_{\theta_1}$ ,$J'_{s_2}\in H_{\theta_{2}}$ and $\gamma
%\in S_{J'_{s_1}}\bigcap S_{J'_{s_2}}$. Since  the equity is true, we
%get   $H_{\theta_1}\bigcap H_{\theta_2}=\emptyset$, thus $s_1\neq
%s_2$, $\gamma \neq \theta_1, \theta_2$. Furthermore, after a \rs \
%$J'_{s_1}{\bigvee}_{\gamma}J'_{s_2}$ the resulting vectors still
%form a cluster.   We lead a contradiction.
%\end{proof}
%\par
%From Lemma \ref{lemma:h=l'} and Eqs. (\ref{equ:mainestimating}) and
%(\ref{inequ:l_0}), we have obtained the desired result for the case
%$t>1$.
\par
If $h=l'$ and $t=1$, which means that ${\theta}_{1}$ appears exactly
one time in each $S_{J_i}(1 \leq i \leq l')$, we will divide the
situation
  into three subcases. %In each one we
%will prove the desired estimation.\\
\newline
\textbf{Case I: $l_0>0$}.\\
Without loss of generality, we suppose $\alpha\in J'_1$ is an
element with single multiplicity in $\bigcup_{i=1}^{l'}{J'_i}$. If
the elements except for $\theta_1$ and $\alpha$ are known,
$\theta_1$ could be determined from the balance of $J'_2$ and then
$\alpha$ could be determined from the balance of $J'_1$. So we have
$$O(b_n^{\frac{\sum_{i=1}^{l'}p'_i-l_0-h_{\theta_1}}{2}+l_0-1})$$
ways of choice  in sum. From $h_{\theta_1}=h=l'$ and Eqs.
(\ref{relation1}) and (\ref{relation2}), we get
\begin{equation*}
O(b_n^{\frac{\sum_{i=1}^{l'}p'_i-l_0-h_{\theta_1}}{2}+l_0-1})=O(b_n^{\frac{\sum_{i=1}^{l}p_i-l}{2}-1})
=o(b_n^{\frac{\sum_{i=1}^{l}p_i-l}{2}}).
\end{equation*}
\\
\textbf{Case II: $l_0=0$ and $m>0$}.\\
As case I above, $\theta_1$ is determined by other elements. To
determine  all the elements other than $\theta_1$, we have
$O(b_n^{\frac{\sum_{}^{}p'_i-l'}{2}})$ ways. From Eqs.
(\ref{relation1}) and (\ref{relation2}), we have
\begin{equation*}
O(b_n^{\frac{\sum_{}^{}p'_i-l'}{2}})=O(b_n^{\frac{\sum_{}^{}p_i-l}{2}-\frac{m}{2}})=
o(b_n^{\frac{\sum_{i=1}^{l}p_i-l}{2}}).
\end{equation*}
\\
\textbf{Case III: $l_0=0$ and $m=0.$}\\
In this case, we cannot do any \emph{reduction} step.
\par
We claim that there exists $\gamma \in S_{J_1}$ other than $\gamma
\neq \theta_1$ such that the number of $+\gamma \,\textrm{and}
-\gamma$ in $\gamma \in S_{J_1}$ are not equal (when $p_{1}$ is even
we can always find some $\gamma$ other than $\theta_1$ such that
$\gamma$ occurs odd times in $S_{J_1}$). Otherwise, ${\theta}_{1}=0$
because of the balance of $J_1$, which contradicts  condition (iii)
in Lemma \ref{general case}.
\par
To  determine  all the elements except for ${\theta}_{1}$ and
$\gamma$, we  have
$$O(b_n^{\frac{(p_1+\ldots+p_l)-(h_{\theta_1}+1)}{2}})$$ ways. Since $H_{\theta_1}=\{J_1,J_2,\ldots,J_{l}\}$,
$H_{\gamma}=$card$|H_{\theta_1}\bigcap H_{\gamma}|=1$, which means
$H_{\gamma}=\{J_1\}$. So we can determine ${\theta}_{1}$ from the
balance of $J_2$. Then $\gamma$ will be  determined by the balance
of $J_1$. Since $h_{\theta_1}=l$, we have
$$O(b_n^{\frac{(p_1+\ldots+p_l)-(h_{\theta_1}+1)}{2}})=o(b_n^{\frac{(p_1+\ldots+p_l)-l}{2}}).$$
\par

 Now we have completed the proof in the case of $l'>1$.

\begin{flushleft}
\Large{$\mathbf{In\: the\: case\:} l'=1$}
\end{flushleft}

We also divide this  case into  two subcases.\\
\textbf{Case I$'$: $l_0>0$}.\\
 Suppose $\alpha\in J'_1$ is an element
with single multiplicity in $\bigcup_{i=1}^{l'}{J'_i}$. If the
elements other than $\alpha$ are known, $\alpha$ could be determined
from the balance of $J'_1$. So we have totally
$O(b_n^{\frac{\sum_{i=1}^{l'}p'_i-l_0}{2}+l_0-1})$ ways. From $l'=1$
and Eqs. (\ref{relation1}), (\ref{relation2}) and (\ref{inequ:l_0}),
we have
\begin{equation*}
O(b_n^{\frac{\sum_{i=1}^{l'}p'_i-l_0}{2}+l_0-1})=O(b_n^{\frac{\sum_{i
=1}^{l}p_i-l}{2}-\frac{m-l_0+1}{2}})=o(b_n^{\frac{\sum_{i=1}^{l}p_i-l}{2}}).
\end{equation*}
\\
\textbf{Case II$'$: $l_0=0$}.\\
In this case, every element in $S_{J'_1}$ appears at least twice.
Thus we have totally $O(b_n^{\frac{\sum_{i=1}^{l'}p'_i}{2}})$ ways.
From $l'=1$ and Eqs.  (\ref{relation1}) and (\ref{relation2}), it
follows that
\begin{equation*}
O(b_n^{\frac{\sum_{i=1}^{l'}p'_i}{2}})=O(b_n^{\frac{\sum_{i=1}^{l}p_i-l}{2}-\frac{l-2}{2}})=
o(b_n^{\frac{\sum_{i=1}^{l}p_i-l}{2}})
\end{equation*}
since $l\geq 3$.

 Now we have completed the proof in the case of
$l'=1$ .

The proof of Lemma  \ref{general case}
 is then complete.
\end{proof}

\begin{remark}\label{remarka0} From the analysis of case III above, we could also understand why the technical but
necessary condition of $a_0\equiv 0$ is assumed in the introduction.
But when $p$ is
even, the assumption of $a_0\equiv 0$ is not necessary.
\end{remark}

\begin{remark}\label{perturbancefactor} From the  calculation of mathematical expectation, variance,
 and higher moments in sections  \ref{expectation}, \ref{variance}, \ref{highermoments},
  the factor $I_{J}$ in the trace formula of  \ref{basic:lem1} can be replaced by a more general
  non-negative integrable bounded function. For example, if $I_{J}\equiv 1$ for all
  $J$, then the results of asymptotic distribution and fluctuation  hold the same as in case
  where
  $b=0$. More precisely, as $n \longrightarrow\iy$ \be{\frac{1}{b^{\frac{p}{2}}_{n}}\sum^{b_{n}}_{j_{1},\ldots,
j_{p}=-b_{n}}\mathbb{E}[\prod_{l=1}^{p}a_{j_{l}}]
\large{\delta}_{0,\sum\limits_{l=1}^{p}j_{l}}}\longrightarrow
\begin{cases} 2^{\frac{p}{2}}(p-1)!!,\
p  \  \text{is even};\\
0, \ \ \ \ \ \ \ \ \ \ \ \ \ \ p \ \text{is odd}
\end{cases}
\ee and the fluctuation can be stated as in Corollary
\ref{corollary} (For conciseness we replace $b_n$ by $n$ there).
From the above point of view and the extended results, our method
gives a mechanism of producing a CLT.
\end{remark}
%-----------------------------------------------------------------end------------------------------------------------------------------------------------

\section{Extensions to other models}\setcounter{equation}{0}
\label{extensions} In this section, we will make use of our method
to deal with some random matrix models closely  related  to Toeplitz
matrices, including Hermitian Toeplitz band matrices, Hankel band
matrices, sparse Toeplitz and Hankel   matrices, generalized
Wishart-type Toeplitz matrices, etc. Besides, we will also  consider
the case of several random matrices in a flavor of free probability
theory. Before we start our extensions, we first review and
generalize the key procedures or arguments in calculating
mathematical expectation (the distribution of eigenvalues),
covariance and higher moments in sections \ref{expectation},
\ref{variance}, \ref{highermoments} respectively   as follows:

1)``Good" trace formulae. See Lemma \ref{basic:lem1}. The simple
formula both represents the form of the matrix and translates our
object of matrix entries to its subscripts $j_{1},\ldots,j_{p}$,
which are integers satisfying some homogeneous equation. In the
present paper, we mainly encounter such homogeneous equations as
\be\label{homogenouseq}\sum\limits_{l=1}^{p}\tau_{l}j_{l}=0,\ee
where $\tau_{l}$ can take $\pm 1$. We write
$J=(j_{1},\ldots,j_{p})$.

2)``Balanced" vectors via some homogeneous equations. We can
generalize the concept of balance: a vector $J=(j_{1},\ldots,j_{p})$
is said to be balanced if its components satisfy one of finite fixed
equations with the form of (\ref{homogenouseq}). From the balance of
a vector, one could determine an element by knowing the other ones
and solve one of the equations.

 3) Reduction via a joint point.  Eliminate the joint point from two correlated vectors
 in some definite way,
 and one  gets a new balanced vector.

 4) Choose a minimal dominating set  by which the freedom degree can be reduced case-by-case.
 We should  particularly be careful about 0.\\

The CLT  is essentially a consequence of the fact that we can omit
the terms who have a cluster consisting of more than two vectors.
Thus if the argument in section \ref{highermoments} is valid, the
CLT is true. In the following models, we will establish a good trace
formulae in each case and then balanced vector and reduction step
can be defined in a natural way. With these equipment the analysis
in section \ref{highermoments} is still valid after a small
adaption.

The mathematical expectation and covariance vary from case to
case. But the way to calculate them is similar with each other. We
will define some suitable integrals associated with pair partitions
similar to those in section \ref{Integrals associated with pair
partitions}, after which we can state the results concisely.

We will only state main results but not give their  proofs in
detail, because the proofs would have been overloaded with
unnecessary notations and minor differences. When necessary, we
might point out some differences of proofs.

\subsection{Hermitian Toeplitz band matrices}

The case  is very similar to real symmetric  Toeplitz band matrices,
except that we now consider $n$-dimensional complex Hermitian
 matrices $T_{n}=(\eta_{ij}\, a_{i-j})_{i,j=1}^{n}$. We assume that
$\mbox{Re}\,a_{j}=\mbox{Re}\,a_{-j}$ and
$\mbox{Im}\,a_{j}=-\mbox{Im}\,a_{-j}$ for $j=1,2,\cdots$, and $
\{\mbox{Re}\,a_{j},\mbox{Im}\,a_{j}\}_{j\in \mathbb{N}}$ is a
sequence of independent real random variables such that \be
\label{hermitiantoeplitz1}\mathbb{E}[a_{j}]=0, \ \
\mathbb{E}[|a_{j}|^{2}]=1 \ \ \textrm{and}\  \mathbb{E}[a_{j}^{2}]=0
\ \textrm{for}\,\ \ j\in \mathbb{N},\ee (homogeneity of 4-th
moments) \be{\label{hermitiantoeplitz2}\kappa=\mathbb{E}[|a_{j}|^{4}],}\ee
 \noindent and further (uniform boundedness)
 \be\label{hermitiantoeplitz3}\sup\limits_{ j\in \mathbb{Z}}
\mathbb{E}[|a_{j}|^{k}]=C_{k}<\iy\ \ \  \textrm{for} \ \ \ k\in
\mathbb{N}.\ee In addition, we also assume $a_{0}\equiv 0$  and  the
bandwidth $b_{n}\ra \iy$ but $b_{n}/n \rightarrow b\in[0,1]$.

\begin{theorem} \label{complextoeplitzcase}With above assumptions and  notations,
 Theorem \ref{tpg} also holds for random Hermitian Toeplitz band
matrices.
\end{theorem}

\begin{remark}The distribution of eigenvalues for this case has been proved to be the same as real case
in \cite{LW}. The covariance like in (\ref{covariance}) is slightly
different from real case because  $\mathbb{E}[a_{j}^{2}]=0$, which
is  given by
\begin{align}
 \sum_{\pi\in \mathcal{P}_{2
}(p,q)}f^{-}_{I}(\pi)+(\kappa-1)\sum_{\pi\in
\mathcal{P}_{2,4}(p,q)}\left(f^{-}_{II}(\pi)+f^{+}_{II}(\pi)\right).
\end{align}
\end{remark}

\subsection{Hankel band matrices} A Hankel matrix $H_{n}=(h_{i+j-1})_{i,j=1}^{n}$
is closely related to a Toeplitz matrix. Explicitly, let
$P_{n}=(\delta_{i-1,n-j})_{i,j=1}^{n}$ the ``backward identity"
permutation, then for a Toeplitz matrix of the form
$T_{n}=(a_{i-j})_{i,j=1}^{n}$ and a Hankel matrix of the form
$H_{n}=(h_{i+j-1})_{i,j=1}^{n}$,   $P_{n}T_{n}$ is a Hankel matrix
and $P_{n}H_{n}$ is a Toeplitz matrix. In this paper we always write
a Hankel band  matrix $H_{n}=P_{n}T_{n}$ where $T_{n}=(\eta_{ij}\,
a_{i-j})_{i,j=1}^{n}$ is a Toeplitz band matrix with bandwidth
$b_{n}$  and the matrix entries $a_{-n+1}, \cdots, a_{0},
\cdots,a_{n-1}$ are real-valued, thus $H_{n}$ is a real symmetric
matrix.

For Hankel band matrices, as in Toeplitz case we  also have a trace
formula and its derivation is similar, see \cite{LW}.
\begin{lemma}\label{hankellemma}
 $\emph{tr}(H_{n}^{p})$=
 \be \label{basic:lem2}\nonumber
 \begin{cases} \sum\limits_{i=1}^{n}\,\sum\limits_{j_{1},\cdots,j_{p}=-b_{n}}^{b_{n}}
\prod\limits_{l=1}^{p}a_{j_{l}}\prod\limits_{l=1}^{p}\chi_{[1,n]}(i-\sum\limits_{q=1}^{l}(-1)^{q}j_{q})
\ \large{\delta}_{0,\sum\limits_{q=1}^{p}(-1)^{q}j_{q}},&\text{$p$ even;}\\
\sum\limits_{i=1}^{n}\,\sum\limits_{j_{1},\cdots,j_{p}=-b_{n}}^{b_{n}}
\prod\limits_{l=1}^{p}a_{j_{l}}\prod\limits_{l=1}^{p}\chi_{[1,n]}(i-\sum\limits_{q=1}^{l}(-1)^{q}j_{q})
\
\large{\delta}_{2i-1-n,\sum\limits_{q=1}^{p}(-1)^{q}j_{q}},&\text{$p$
odd.}
\end{cases}
%\sum_{i=1}^{N}\,\sum_{j_{1},\cdots,j_{k}=-b_{N}}^{b_{N}}
%\prod_{l=1}^{k}a_{j_{l}}\prod_{l=1}^{k}I_{[1,N]}(i+\sum_{q=1}^{l}j_{q})
%\ \large{\delta}_{0,\sum\limits_{q=1}^{k}j_{q}}, k\in \mathbb{N}.
 \ee
\end{lemma}

From the above trace formula, our method can apply to the case that
$p$ is even. We consider random Hankel matrices satisfying the
following assumptions: assuming that $\{a_{j}:j\in \mathbb{Z}\}$ is
a sequence of
independent real  random variables %(implying that all odd moments
%vanish)
 such that

\be \label{symmetrichankel1} \mathbb{E}[a_{j}]=0, \ \
\mathbb{E}[|a_{j}|^{2}]=1 \ \ \textrm{for}\,\ \ j\in \mathbb{Z},\ee
(homogeneity of 4-th moments)
\be{\label{symmetrichankel2}\kappa=\mathbb{E}[|a_{j}|^{4}], j\in
\mathbb{Z}}\ee
 \noindent and further (uniform boundedness)\be\label{symmetrichankel3}\sup\limits_{j\in \mathbb{Z}}
 \mathbb{E}[|a_{j}|^{k}]=C_{k}<\iy\ \ \
\textrm{for} \ \ \ k\geq 3.\ee %Notice that Chatterjee
%\cite{chatterjee} has pointed,
In addition, we also assume  the bandwidth $b_{n}\ra \iy$ but
$b_{n}/n
\rightarrow b\in [0,1]$ as $n\rightarrow \infty$.

\begin{theorem}\label{hankeltheorem}
Let $H_{n}$ be  a real symmetric
((\ref{symmetrichankel1})--(\ref{symmetrichankel3})) random Hankel
band matrix with the bandwidth $b_{n}$, where $b_{n}/n \rightarrow b
\in [0,1]$ but $b_{n}\ra \iy$ as $n\rightarrow \infty$. Set
$A_{n}=H_{n}/\sqrt{b_{n}}$ and
\be{\zeta_{p}=\frac{\sqrt{b_{n}}}{n}\left(\mathrm{tr}(A_{n}^{2p})-\mathbb{E}[\mathrm{tr}(A_{n}^{2p})]\right).}\ee
Then \be \zeta_{p}\longrightarrow N(0,\tilde{\sigma}_p^2)\ee in
distribution as $n\rightarrow \infty$. Moreover, for a given
polynomial \be{Q(x)=\sum_{j=0}^{p}q_{j} x^{j}}\ee with degree $p\geq
1$, set
\be{\zeta_{Q}=\frac{\sqrt{b_{n}}}{n}\left(\mathrm{tr}Q(A^{2}_{n})-\mathbb{E}[\mathrm{tr}Q(A^{2}_{n})]\right),}\ee
we also have\be \zeta_{Q}\longrightarrow N(0,\tilde{\sigma}_Q^2)\ee
in distribution as $n\rightarrow \infty$. Here the variances
$\tilde{\sigma}_p^{2}$ and $\tilde{\sigma}_Q^2$ will be given below.%\ref{variance}.
\end{theorem}

We remark that the limit of
$\frac{1}{n}\mathrm{tr}(H_{n}/\sqrt{b_{n}})^{p}$ can be calculated
in the same way as in Toeplitz case, see also \cite{LW}. Next, we
derive briefly the variances $\tilde{\sigma}_p^{2}$ and
$\tilde{\sigma}_Q^2$.

By Lemma \ref{hankellemma}, rewrite
 \begin{equation} \zeta_{p}=\frac{1}{n b_{n}^{\frac{2p-1}{2}}}
 \sum_{i=1}^{n}\,\sum_{J}I_{J}\left(a_{J}-
\mathbb{E}[a_{J}]\right) \
\large{\delta}_{0,\sum\limits_{l=1}^{2p}(-1)^{l}j_{l}}.
 \end{equation}
Here $J=(j_{1},\ldots,j_{2p})\in \{-b_{n},\ldots,b_{n}\}^{2p}$,
$a_{J}=\prod^{2p}_{l=1}a_{j_{l}}$,
$I_{J}=\prod^{2p}_{k=1}\chi_{[1,n]}(i-\sum_{l=1}^{k}(-1)^{l}j_{l})$
and the summation $\sum_{J}$ runs over all possibile $J \in
\{-b_{n},\ldots,b_{n}\}^{2p}$. We call a vector
$J=(j_{1},\ldots,j_{2p})\in \{-b_{n},\ldots,b_{n}\}^{2p}$ is
balanced if
\be\label{balancedhankel}\sum\limits_{l=1}^{2p}(-1)^{l}j_{l}=0.\ee
Since $a_j$ and $a_{-j}$ are independent when $j\neq 0$, $S_J$, the
projection of $J$, should be $J$ itself (forget the order of its
components).
\par
The only difference is the definition of balance and projection.
Using the same technique in section \ref{variance}, we get
\begin{align}
 \mathbb{E}[\zeta_{p}\,\zeta_{q}]\longrightarrow \tilde{\sigma}_{p,q}=
\sum_{\pi\in \mathcal{P}_{2
}(2p,2q)}g_{I}(\pi)+(\kappa-1)\sum_{\pi\in \mathcal{P}_{2,4
}(2p,2q)}g_{II}(\pi)
\end{align}
as $n\rightarrow \infty$.

Here for a fixed pair partition $\pi$, we construct a projective
relation between two groups of unknowns ${y_{1},\ldots,y_{2p+2q}}$
and ${x_{1},\ldots,x_{p+q}}$ as follows:
\be{\label{unknownsrelationhankel}y_{i}=y_{j}=x_{\pi(i)} }\ee
whenever $i\thicksim_{\pi} j$.
\par
If $\pi \in \mathcal{P}_{2
}(2p,2q)$, let
\begin{align}
\label{typeIhankel} &g_{I}(\pi)=\int_{[0,1]^{2}\times [-1,1]^{p+q}}
\delta\left(\sum_{i=1}^{2p}(-1)^{i}y_{i}\right)
\chi_{ \left\{\sum\limits_{i=2p+1}^{2p+2q}(-1)^{i}y_{i}=0 \right\}}\\
&\prod_{j=1}^{2p}\chi_{[0,1]}(x_{0}-b\sum_{i=1}^{j}(-1)^{i}y_{i})
\prod_{j'=2p+1}^{2p+2q}\chi_{[0,1]}(y_{0}-b\sum_{i=2p+1}^{j'}(-1)^{i}y_{i})\,d\,y_{0}
\prod_{l=0}^{p+q}d\, x_{l}\nonumber.
\end{align}\\
\par
If $\pi \in \mathcal{P}_{2,4}(2p,2q)$, let
\begin{align}\label{typeIIhankel} &g_{II}(\pi)=
 \int_{[0,1]^{2}\times
[-1,1]^{p+q-1}}\chi_{ \left\{\sum\limits_{i=1}^{2p}(-1)^{i}y_{i}=0
\right\}}\chi_{ \left\{\sum\limits_{i=2p+1}^{2p+2q}(-1)^{i}y_{i}=0
\right\}} \displaystyle
%\sum_{k=1,k\neq \pi(i_{0})
%}^{p+q}\delta\left(x_{\pi(i_{0})}-
% x_{k}\right)
 \\& \prod_{j=1}^{2p}\chi_{[0,1]}(x_{0}-b\sum_{i=1}^{j}(-1)^{i}y_{i})
\prod_{j'=2p+1}^{2p+2q}\chi_{[0,1]}(y_{0}-b\sum_{i=2p+1}^{j'}(-1)^{i}y_{i})\,d\,y_{0}
\prod_{l=0}^{p+q-1}d\, x_{l}.\nonumber\end{align}

%\begin{remark}\label{typeIIhankelremark}
%\par
Because of the existence of the  characteristic  function in the
integrals of type I and  II, we see that $g_{I}(\pi)\neq 0$ if and
only if $\sum\limits_{i=2p+1}^{2p+2q}(-1)^{i}y_{i}\equiv 0$ when
$\sum\limits_{i=1}^{2p}(-1)^{i}y_{i}=0$. Denote the subset of
$\mathcal{P}_{2}(2p,2q)$ consisting of this kind of $\pi$ by
$\mathcal{P}^{I}_{2}(2p,2q)$.
\par
$g_{II}(\pi)\neq 0$ if and only if
$\sum\limits_{i=2p+1}^{2p+2q}(-1)^{i}y_{i}\equiv
\sum\limits_{i=1}^{2p}(-1)^{i}y_{i}\equiv 0$. Denote the subset of
$\mathcal{P}_{2,4}(2p,2q)$ consisting of this kind of $\pi$ by
$\mathcal{P}^{II}_{2,4}(2p,2q)$. From the definition of balance, if
we denote $V_i=\{i_1,i_2,i_3,i_4\}$ to be the block with four
elements and $1\leq i_1<i_2\leq 2p<i_3<i_4\leq 2q$, then $i_1+i_2$
and $i_3+i_4$ must be odd. Moreover, $j+k$ is odd provided
$j,k\not\in V_i$ and $i\thicksim_{\pi} j$.
%there exist exactly two
%crossing matches of $\pi$ which both take place in $V_i$. Thus $\pi$
%can be decomposed into to an element $\pi_1 \in \mathcal{P}_{2}(2p)$
%and another one $\pi_2 \in\mathcal{P}_{2}(2q)$. From the definition
%of balance, $i+j$ is even if $i\thicksim_{\pi_1} j$ or
%$i\thicksim_{\pi_2} j$.
%\end{remark}
\par
From the discussion above, we get
\begin{align}
 \mathbb{E}[\zeta_{p}\,\zeta_{q}]\longrightarrow
 \tilde{\sigma}_{p,q} =
\sum_{\pi\in \mathcal{P}^{I}_{2
}(2p,2q)}g_{I}(\pi)+(\kappa-1)\sum_{\pi\in \mathcal{P}^{II}_{2,4
}(2p,2q)}g_{II}(\pi),
\end{align}
\begin{equation}
  \tilde{\sigma}_p^{2}=\tilde{\sigma}_{p,p}
\end{equation}
and \begin{equation}
 \tilde{\sigma}_Q^2 =\sum_{i=1}^{p}\sum_{j=1}^{p}q_{i}q_{j}
\tilde{\sigma}_{i,j}.
\end{equation}

\par
Similar to Corollary \ref{corollary}, we get another central limit
theorem for product of independent random variables.
\begin{coro}\label{corollary} Suppose that
$\{a_{j}:j\in\mathds{Z}\}$  is a sequence of independent random
variables satisfying the assumptions
(\ref{symmetrichankel1})--(\ref{symmetrichankel3}). For every $p\geq
1$, \be{\frac{1}{n^{\frac{2p-1}{2}}}\sum^{n}_{ j_{1},\ldots,
j_{2p}=-n}\left(\prod_{l=1}^{2p}a_{j_{l}}-\mathbb{E}[\prod_{l=1}^{2p}a_{j_{l}}]\right)
\large{\delta}_{0,\sum\limits_{l=1}^{2p}(-1)^{l}j_{l}}}\ee converges
in distribution to a Gaussian distribution $N(0,\tilde{\sigma}_p^2)$. Here
the variance $\tilde{\sigma}_p^{2}$ as above corresponds to the case $b=0$.
\end{coro}

\begin{remark}\label{rem:even length}
The derivation of higher moments can be calculated in the same way
as in section \ref{highermoments}. Careful examination shows  that
the argument in case III in section \ref{highermoments} (the
existence of $\gamma$) is the only part depending on the definition
of balance. However, since $2p$ and $2q$ are even, the dimensions of
all vectors involved are even. Remark \ref{remarka0} still applies
so that one needn't care much about the definition of balance.
Moreover, $a_0$ needn't equal to $0$. The same thing happens in
section \ref{sub:section:singular value}.
\end{remark}

\subsection{Sparse Toeplitz and Hankel matrices}

A sparse matrix is a matrix in which some entries are replaced by 0,
which occurs in some application background. A sparse random matrix
provides a more natural and relevant description of the complex
system in nuclei physics, we refer to \cite{BS} for an introduction.

A sparse Toeplitz matrix can be expressed as follows. Let
$$B_{n}=\left(\bigcup_{j=1}^{r} [b_{n}^{(2j-1)},b_{n}^{(2j)}]\right)\bigcap
\{0,1,\ldots,n-1\}$$% be a subset of $\{0,1,\ldots,n-1\}$,
we define
\be{\xi_{ij}=
 \begin{cases} 1,\ \ |i-j|\in B_{n};\\
0,\ \ \text{otherwise.}
\end{cases}
}\ee Then a sparse Toeplitz matrix is defined by
\be{\label{sparsematrices}T_{n}=(\xi_{ij}\,
a_{i-j})_{i,j=1}^{n},}\ee and the corresponding Hankel matrix is
$H_{n}=P_{n}T_{n}$ as before. Note that for Hankel matrices we can
define a more general case with minor adaptations:
$$B_{n}=\left(\bigcup_{j=1}^{r} [b_{n}^{(2j-1)},b_{n}^{(2j)}]\right)\bigcap
\{0,\pm1,\ldots,\pm(n-1)\}$$% be a subset of $\{0,1,\ldots,n-1\}$,
and  \be{\xi_{ij}=
 \begin{cases} 1,\ \ i-j\in B_{n};\\
0,\ \ \text{otherwise.}
\end{cases}
}\ee
 In
addition, we assume that there exists a sequence $\{b_{n}\}$with
$b_{n}\ra \iy$ but $b_{n}/n \rightarrow b\in[0,1]$ such that
$\frac{b_{n}^{(j)}}{b_{n}} \ra b^{(j)}$ as $n \ra \iy$. Set
 \be
B_{+}=\bigcup_{j=1}^{r}[b^{(2j-1)},b^{(2j)}]\subseteq [0,1]\ee %in
%distribution as $n \ra \iy$. Here $B_{+}$ denotes the union of
%finite intervals.

With the above assumptions, we claim:
\begin{theorem} Under the same conditions therein,
 Theorems \ref{tpg} and \ref{complextoeplitzcase} also hold  for
 sparse random Toeplitz matrices, and Theorem \ref{hankeltheorem} also holds
  for sparse random Hankel matrices.
\end{theorem}

We remark that in sparse matrix case, the only difference is
integral interval of variables, for example, the moments in
 Eq. (\ref{bandtoeplitz:moment}) become
\be M_{2k}=\sum_{\pi\in \mathcal{P}_{2 }(2k)}\int_{[0,1]\times
B^{k}}\prod_{j=1}^{2k}\chi_{[0,1]}(x_{0}+b\sum_{i=1}^{j}\epsilon_{\pi}(i)\,x_{\pi(i)})
\prod_{l=0}^{k}\mathrm{d}\, x_{l},\ee where $B=B_{+}\bigcup B_{-}$,
while $B_{-}:=\{x|-x\in B_{+}\}$.

\subsection{Singular values of powers of  Toeplitz matrices}\label{sub:section:singular value}
For a real or complex (random) Toeplitz band matrix
$T_{n}=(\eta_{ij}\, a_{i-j})_{i,j=1}^{n}$ with the bandwidth
$b_{n}$, writing $T^{*} _{n}$ for the adjoint matrix of $T_{n}$,  we
consider the product $T^{\ast\,s}_{n}T^{s}_{n}$ of powers with some
$s\in \{1,2,\ldots,\}$. When the Toeplitz $T_{n}$ is replaced by a
$n \times n$ matrix, the norm of product of random matrices in the
above form was studied in \cite{BY}, when investigating the limiting
behavior of solutions to large systems of linear equations.
Recently, the limiting distribution of the product of random
matrices has been  studied in \cite{AGT} and its moments is known as
Fuss-Catalan numbers.

Set \be W_{n}^{(s)}=\frac{1}{b^{s}_{n}}T^{\ast\,s}_{n}T^{s}_{n}. \ee

For real Toeplitz case,  we assume that $\{a_{j}:j\in \mathbb{Z}\}$
is a sequence of independent real random variables satisfying
((\ref{symmetrichankel1})--(\ref{symmetrichankel3})). For complex
Toeplitz case,  we assume that  $
\{\mbox{Re}\,a_{j},\mbox{Im}\,a_{j}\}_{j\in \mathbb{Z}}$ is a
sequence of independent real random variables such that \be
\mathbb{E}[a_{j}]=0, \ \ \mathbb{E}[|a_{j}|^{2}]=1 \ \ \textrm{and}\
\mathbb{E}[a_{j}^{2}]=0 \ \textrm{for}\,\ \ j\in \mathbb{Z},\ee
(homogeneity of 4-th moments)
\be{\kappa=\mathbb{E}[|a_{j}|^{4}],}\ee
 \noindent and further (uniform boundedness)
 \be\sup\limits_{ j\in \mathbb{Z}}
\mathbb{E}[|a_{j}|^{k}]=C_{k}<\iy\ \ \  \textrm{for} \ \ \ k\in
\mathbb{N}.\ee In addition, we also assume  the bandwidth $b_{n}\ra
\iy$ but $b_{n}/n \rightarrow b\in[0,1]$ as $n\rightarrow \infty$.
Note that we have the trace formulas as in Eq. (\ref{basic:lem1})
(more like the case of even $p$ in Eq. (\ref{basic:lem2}))
\begin{equation}
 \mathrm{tr}\left( (W_{n}^{(s)})^{p}\right)=\frac{1}{b^{p s}_{n}}\sum_{i=1}^{n}\,\sum_{J}
a_{J}\,I_{J} \
\large{\delta}_{0,\sum\limits_{l=1}^{2ps}(-1)^{[\frac{l-1}{s}]_{0}}j_{l}},\quad
\quad p\in \mathbb{N}.
 \end{equation}
  Here
$J=(j_{1},\ldots,j_{2ps})\in \{-b_{n},\ldots,b_{n}\}^{2ps}$,
$I_{J}=\prod^{2ps}_{k=1}\chi_{[1,n]}(i+\sum_{l=1}^{k}(-1)^{[\frac{l-1}{s}]_{0}}j_{l})$
and $a_{J}=\prod^{2ps}_{l=1}\tilde{a}_{j_{l}}$. $[x]_{0}$ denotes
the largest integer not larger than $x$, and \be{\tilde{a}_{j_{l}}=
 \begin{cases} a_{j_{l}},\ \text{if} \ [\frac{l-1}{s}]_{0}\ \text{is even};\\
\bar{a}_{j_{l}},\ \text{otherwise.}
\end{cases}
}\ee On the other hand, we remark that when $s=1$ the  real Toeplitz
case reads
\begin{equation}
 \mathrm{tr}\left(
 (W_{n}^{(1)})^{p}\right)=\frac{1}{b^{p}_{n}}\mathrm{tr}(H_{n}^{2p}),
 \end{equation}
where $H_{n}$ denotes the Hankel matrix as in Theorem
\ref{hankeltheorem}. Analogous to Theorem \ref{hankeltheorem}, we
have the following theorem for real and complex Toeplitz matrices:
\begin{theorem} Under the above assumptions, for
 a given polynomial \be{Q(x)=\sum_{j=0}^{p}q_{j} x^{j}}\ee with
degree $p\geq 1$, then \be\frac{\sqrt{b_{n}}}{n}
\left(\mathrm{tr}Q(W_{n}^{(s)})-\mathbb{E}[\mathrm{tr}Q(W_{n}^{(s)})]\right)
\longrightarrow N(0,\hat{\sigma}_Q^2) \ee in distribution as
$n\rightarrow \infty$. Here
 $\hat{\sigma}_Q^2$ denotes  the variance.
\end{theorem}

\begin{remark}
Because all the involved vectors have even dimensions, according to
Remark \ref{rem:even length}, the assumption of $a_0=0$ is not
necessary.
\end{remark}

We can evaluate the variance $\hat{\sigma}_Q^2$ as in the Hankel
case, but its expression is very redundant so we omit it. Here we
only gives the limit of $p$-moment, i.e.,
\begin{align}
 &\mathrm{E}\left[\frac{1}{n}\mathrm{tr}\left( (W_{n}^{(s)})^{p}\right)\right]\longrightarrow \nonumber\\
 &\sum_{\pi\in \mathcal{P}_{2}(2ps)}\int_{[0,1]\times
[-1,1]^{ps}}\chi_{\{\sum_{i=1}^{2ps}(-1)^{[\frac{l-1}{s}]_{0}}y_{l}=0\}}\prod_{j=1}^{2ps}\chi_{[0,1]}(x_{0}+b\sum_{i=1}^{j}(-1)^{[\frac{l-1}{s}]_{0}}y_{l})
\prod_{l=0}^{ps}\mathrm{d}\, x_{l}.\label{wishat-typemoment}
 \end{align}
Here for a given partition $\pi\in \mathcal{P}_{2}(2ps)$, the
unknowns $\{y_{1},\ldots,y_{2ps}\}$ and $\{x_{1},\ldots,x_{ps}\}$
satisfies the relation: \be{y_{i}=y_{j}=x_{\pi(i)} }\ee whenever
$i\thicksim_{\pi} j$.

\begin{remark} The non-zero terms in the summation of
Eq. (\ref{wishat-typemoment}) come from such partitions as\be
\label{alternative}
(-1)^{[\frac{i-1}{s}]_{0}}+(-1)^{[\frac{j-1}{s}]_{0}}=0\ee whenever
$i\thicksim_{\pi} j$. We can treat the desired pair partitions as
follows: considering   $2p$ alternative groups
$$\overbrace{1,1,\ldots,1}^{s},\overbrace{-1,-1,\ldots,-1}^{s},\ldots,
\overbrace{1,1,\ldots,1}^{s},\overbrace{-1,-1,\ldots,-1}^{s}$$ each
of which consists of $s$ 1's  or $s$ (-1)'s. Then the pair
partitions satisfying (\ref{alternative}) correspond to the pair
matches of these $2ps$ 1 or -1  such that each pair has one 1 and
one -1. The total number of this kind of partitions is $(ps)!$.
Thus, in the special case  $b=0$, the $p$-th moment in
(\ref{wishat-typemoment}) equals $2^{ps}(ps)!$. By the property of
$\Gamma$- function, one knows that $2^{ps}(ps)!$ is the $p$-th
moment of the density
$f(x)=\frac{1}{2s}x^{\frac{1-s}{s}}\exp(-\frac{1}{2}x^{\frac{1}{s}})$
on $[0,\iy)$. However, unfortunately when $s>1$ the density $f(x)$
is not uniquely determined by its moments, see \cite{feller}. Thus
we cannot say that the distribution of eigenvalues of $W_n^{(s)}$
converges weakly to $f(x)$.
\end{remark}

\subsection{Product of several Toeplitz matrices}
In free probability theory, one usually  considers the limit of
joint moments of several independent matrices \cite{hp,ns}. Based on
this point of view, the first version \cite{LW0} of \cite {LW} by
two of the  authors,  provided the limit joint distribution  of
several independent Toeplitz and Hankel matrices. In the present
paper, we will prove Gaussian fluctuation of the joint moments for
symmetric  Toeplitz matrices (the same result proves right for
Hermitian Toeplitz case and Hankel matrices ). More precisely, given
independent symmetric Toeplitz band matrices $T_{1,n}, T_{2,n},
\ldots, T_{r,n}$, we study the limit and fluctuation of
$$\mathrm{tr}(T_{i_{1},n} T_{i_{2},n} \cdots T_{i_{p},n})$$
for $1\leq i_{1},i_{2},\ldots,i_{p}\leq r$.

First of all, we introduce some notations.  For given $1\leq i_{1},
\cdots, i_{p}, i_{p+1}, \cdots, i_{p+q}\leq r$, there exists an
associated partition $\pi_{0}=\{V_{1},\cdots,V_{s}\}$ of $[p+q]$
such that \be j\thicksim_{\pi_{0}}k \Longleftrightarrow
i_{j}=i_{k}.\ee Let $\mathcal{P}^{0}_{2}(p,q)$
 denote a subset of $\mathcal{P}_{2}(p,q)$, which consists of such pair partitions
$\pi$: \be j\thicksim_{\pi}k \Longrightarrow j\thicksim_{\pi_{0}}k
.\ee When  $q=0$, we denote $\mathcal{P}^{0}_{2}(p,0)$ by
$\mathcal{P}^{0}_{2}(p)$. Similarly,  $\mathcal{P}^{0}_{2,4}(p,q)$
 denotes a subset of $\mathcal{P}_{2,4}(p,q)$, which consists of such   partitions
$\pi$: \be j\thicksim_{\pi}k \Longrightarrow j\thicksim_{\pi_{0}}k
.\ee

\begin{theorem}
Let $T_{1,n}, T_{2,n}, \ldots, T_{r,n}$ be independent  real
symmetric random Toeplitz band matrices, each of which satisfies the
assumptions in Theorem \ref{tpg}. with the above notations, we have

  \begin{align}&\frac{1}{n\,b^{\frac{p}{2}}_{n}}
 \mathbb{E}[\mathrm{tr}(T_{i_{1},n} T_{i_{2},n} \cdots
 T_{i_{p},n})]\longrightarrow \nonumber\\
 &=\begin{cases}\sum\limits_{\pi\in \mathcal{P}^{0}_{2 }(p)}\int_{[0,1]\times
[-1,1]^{p/2}}\prod\limits_{j=1}^{p}\chi_{[0,1]}(x_{0}+b\sum_{i=1}^{j}\epsilon_{\pi}(i)\,x_{\pi(i)})
\prod\limits_{l=0}^{p/2}\mathrm{d}\, x_{l},\ \mathrm{even}\  p;\\
0, \ \ \mathrm{odd} \ p
\end{cases}\nonumber
\end{align}
as $n \longrightarrow  \iy$. Now, set \be S_{n}(i_{1},i_{2}, \ldots,
i_{p})=\frac{\sqrt{b_{n}}}{n}\frac{1}{b^{\frac{p}{2}}_{n}}\left(\mathrm{tr}(T_{i_{1},n}
T_{i_{2},n} \cdots
 T_{i_{p},n})-
 \mathbb{E}[\mathrm{tr}(T_{i_{1},n} T_{i_{2},n} \cdots
 T_{i_{p},n})]\right).\ee
Then for $p\geq 2$ the family $\{S_{n}(i_{1},i_{2}, \ldots, i_{p}):
1\leq i_{1}, \cdots, i_{p}\leq r\}$ converges in distribution to a
Gaussian family. For fixed $p,\ q \geq 2$, we have
\begin{align}
 &\mathbb{E}[S_{n}(i_{1},i_{2}, \ldots,
i_{p})\,S_{n}(i_{p+1},i_{p+2}, \ldots, i_{p+q})]\longrightarrow \nonumber\\
& \sum_{\pi\in \mathcal{P}^{0}_{2
}(p,q)}\left(f^{-}_{I}(\pi)+f^{+}_{I}(\pi)\right)+(\kappa-1)\sum_{\pi\in
\mathcal{P}^{0}_{2,4
}(p,q)}\left(f^{-}_{II}(\pi)+f^{+}_{II}(\pi)\right)
\end{align}
when $p+q$ is even and
\begin{align}
 \mathbb{E}[S_{n}(i_{1},i_{2}, \ldots,
i_{p})\,S_{n}(i_{p+1},i_{p+2}, \ldots, i_{p+q})]=o(1)
\end{align}
when $p+q$ is odd as $n\longrightarrow \infty$.

\end{theorem}

% Note that in
%order to obtain Gaussian fluctuation of random Toeplitz
% matrices, we must  whose elements with general random (we will explain why we do, see below!)

%\vspace{.2cm}
%
%\noindent\textbf{An added note.} After the paper as submitted we
%learned from the Associate Editor and the referee about a related
%preprint ``Limiting Spectral Distribution of Some Band Matrices" by
%Basak and Bose at the site
%http://www.isical.ac.in/~statmath/html/publication/techreport.html.
%Although their paper and ours contain the same results for Hankel
%and Toeplitz band matrices, the former assumes less integrability on
%the entries of a matrix, allows more general ``rates" for the
%bandwidth, and also covers more ensembles of ``structured matrices"
%related to Toeplitz matrices. On the other hand, our paper covers
%Hermitian Toeplitz matrices and  gives low order moment
%calculations. Besides, our paper gives a different method for
%analyzing Toeplitz matrices by treating them as a linear combination
%of deterministic matrices with independent coefficients. This method
% can be used to derive
%other results, including those that deal with semicircle law.
%
%\section*{Acknowledgements}
%The authors thank the Associate Editor for valuable comments, the
%referees for very helpful comments and careful writing instruction
%on the draft of this paper, and Steven Miller for information on his
%paper.

%\newpage


\begin{thebibliography}{99}
\bibitem{AGT} Alexeev, N., G$\ddot{o}$tze,  F.   and  Tikhomirov, A.: Asymptotic distribution of singular values of
powers of random matrices. arXiv: 1002.4442v1

\bibitem{AZ} { Anderson, G.} and { Zeitouni, O.}:
A CLT for a band matrix model. {\it Probab. Theory Related Fields}
{\bf 134} no. 2, 283--338 (2006).


\bibitem{bai} Bai, Z. : Methodologies in
spectral analysis of large-dimensional random matrices, a review,
\textit{Statist.\ Sinica.} \textbf{9} (1999), 611--677.
\bibitem{BS} Bai, Z.,  Silverstein, J. W. : Spectral analysis of large dimensional  random matrices.
 Science Press, Beijing,  2006


\bibitem{BY} Bai, Z.,   Yin, Y.Q.: Limiting behavior of the norm of products of random
matrices and two problems of Geman-Hwang \textit{ Probab. Theory
Related Fields}\textbf{73}(1986), 555--569.




\bibitem{bb}  Basak, A.  and Bose, A.:  Limiting spectral distribution of some band matrices, preprint
2009.

\bibitem{basor} Basor, E. L. :  Toeplitz determinants, Fisher--Hartwig symbols and random matrices,
\textit{Recent Perspectives in Random Matrix Theory and Number
Theory, London Math. Soc. Lecture Note Ser. } \textbf{322} (2005),
309--336.

%\bibitem{bmp} L. V. Bogachev, S. A. Molchanov and L. A. Pastur,  On the level density of random band matrices,
%\textit{Math. Notes.} \textbf{50} (1991), 1232--1242.
\bibitem{bcg}  Bose, A., Chatterjee, S. and Gangopadhyay, S. :  Limiting spectral distribution of large dimensional
random matrices, \textit{J. Indian Statist. Assoc.} \textbf{41}
(2003), 221--259.

\bibitem{bm} Bose, A. ,Mitra, J. :  Limiting spectral distribution of a special
circulant, \textit{Stat. Probab. Letters} \textbf{60} (2002), no.1,
111--120.

%\bibitem{bg} A.~B$\ddot{o}$ttcher and S. M. Grudsky, {\it Spectral Properties of Banded Toeplitz Matrices\/},
%Society for Industrial Mathematics, Philadelphia, 2005.

\bibitem{bdj}  Bryc, W., Dembo, A. and Jiang,  T.: Spectral measure of large random Hankel, Markov and Toeplitz
matrices, \textit{Ann.\ Probab.} \textbf{34} (2006), no.1, 1--38.

%\bibitem{cmi} G. Casati, L. Molinari and F. Izrailev, Scaling properties of band random matrices,
% \textit{Phys. Rev. Lett.} \textbf{64} (1990), 1851--1854.







\bibitem{chatterjee}Chatterjee, S.: Fluctuations of eigenvalues and second order
Poincar$\acute{e}$ inequalities, Probab. Theory Related Fields 143,
1--40 (2009).

%\bibitem{chatterjee} S.
%Chatterjee, Fluctuations of eigenvalues and second order Poincare
%inequalities, \textit{ Probab. Theory Related
%Fields}\textbf{143}(2009), 1--40.
\bibitem{diaconis} Diaconis, P. : Patterns in eigenvalues: The 70th Josiah Willard Gibbs
lecture,  \textit{Bull. Amer. Math. Soc.} \textbf{40} (2003),
115--178.

\bibitem{DEvans} {Diaconis, P.} and { Evans, S.~N.}: Linear functionals of eigenvalues of random
matrices. {\it Trans. Amer. Math. Soc.} {\bf 353} 2615--2633(2001).

\bibitem{Dsha} {Diaconis, P.} and { Shahshahani, M. }:
 On the eigenvalues of random matrices: Studies in applied
probability. {\it J. Appl. Probab.} {\bf 31A}, 49--62(1994).

\bibitem{DE} Dumitriu, I.  and Edelman, A.:
 Global spectrum fluctuation for the $\beta$-Hermite and $\beta$-Laguerre ensembles
via matrix models, \textit{Journal of  Mathematical Physics}
\textbf{47}(6), 063302--063336(2006)

\bibitem{feller} Feller, W.~: {\it An Introduction to Probability Theory and Its
Applications\/}, Volum 2, Wiley, New York, 1971.



\bibitem{GS} Grenander, U.,  Szeg$\ddot{o}$, G.:
{\it Toeplitz forms and their applications}. University of
California Press, Berkeley-Los Angeles, 1958.

\bibitem{hm} Hammond, C.~  and Miller, S. J.:  Distribution of eigenvalues for the ensemble of
real symmetric Toeplitz matrices, \textit{J.\ Theoret.\ Probab.}
\textbf{18} (2005), 537--566.

\bibitem{hp} Hiai, F.~ and Petz,  D.:  {\it The Semicircle Law, Free Random Variables and Entropy\/},
Amer. Math. Soc., Providence, RI, 2000.



\bibitem{johansson1} {Johansson, K.}: On  Szego's asymptotic formula for Toeplitz determinants
and generalizations. Bull. Sci. Math. (2)  {\bf 112},  257--304
(1988).
%\bibitem{johansson2} {Johansson, K.}: On random matrices from the classical compact groups.
%{\it Ann. Math. (2)} {\bf 145}, 519--545 (1997).


\bibitem{johansson3} {Johansson, K.}:
 On the fluctuation of eigenvalues of random Hermitian matrices. {\it Duke Math. J.} {\bf 91}, 151--204 (1998).


\bibitem{jonsson} Jonsson, D.: Some limit theorems for the eigenvalues of a sample covariance
matrix. {\it J. Mult. Anal.} {\bf 12},  1--38(1982).

\bibitem{kargin} Kargin, V.: Spectrum  of random Toeplitz matrices with band structures,
 Elect. Comm. in Probab. 14 (2009), 412--421





%\bibitem{kk} A. Khorunzhy and W. Kirsch, On asymptotic expansions and scales of spectral universality in
%band random matrix ensembles, \textit{Commun. Math. Phys.}
%\textbf{231} (2002), 223--255.
%
%\bibitem{mehta1} M.~L.~Mehta, {\it Matrix Theory: Selected Topics and Useful Results\/},
% Hindustan Publising Corporation, 1989.
%
%
%\bibitem{mehta2} M.~L.~Mehta, {\it Random Matrices\/},
%3nd ed.,  Academic Press, San Diego, 2004.

%\bibitem{mpk} S. A. Molchanov, L. A. Pastur and A. M. Khorunzhy, Eigenvalue distribution for band random matrices
%in the limit of their infinite rank, \textit{Teor. Matem. Fizika.}
%\textbf{90} (1992), 108--118.



















%\bibitem{fm} Y. V. Fyodorov and A. D. Mirlin, Scaling properties of localization in random band matrices: A $\sigma$-model
%approach, \textit{Phys. Rev. Lett.} \textbf{67} (1991), 2405--2409.




\bibitem{hj} Horn, R.~A.~, Johnson, C. R.: {\it Matrix Analysis\/},  Cambridge University Press, New York, 1985.

%\bibitem{hp} F.~Hiai and D. Petz, {\it The Semicircle Law, Free Random Variables and Entropy\/},
%Amer. Math. Soc., Providence, RI, 2000.

%\bibitem{kk} A. Khorunzhy and W. Kirsch, On asymptotic expansions and scales of spectral universality in
%band random matrix ensembles, \textit{Commun. Math. Phys.}
%\textbf{231} (2002), 223--255.
\bibitem{LW0} Liu, D.-Z.,  Wang, Z.-D.: Limit Distributions for Random Hankel, Toeplitz Matrices and Independent
Products, arXiv:0904.2958v1

\bibitem{LW} Liu, D.-Z.,  Wang, Z.-D.: Limit Distribution of Eigenvalues for Random Hankel and Toeplitz Band Matrices,
\textit{Journal of Theoretical Probability},
10.1007/s10959-009-0260-4 (2009)
\bibitem{mms} Massey, A.,  Miller, S. J. and Sinsheimer,  J.: Distribution of Eigenvalues of Real Symmetric
Palindromic Toeplitz Matrices and Circulant Matrices, \textit{J.\
Theoret.\ Probab.} \textbf{20} (2007), 637--662.

\bibitem{ns} Nica, A.~ and Speicher, R.: {\it Lectures on the Combinatorics of Free Probability\/},
Cambridge University Press,  2006.


\bibitem{popescu} { Popescu, I.}:
 General tridiagonal random matrix models, limiting distributions
and fluctuations. {\it Probab. Theory Related Fields} {\bf 114},
179--220 (2009).

\bibitem{SS} Sinai, Y. and Soshnikov, A.:  Central Limit Theorem for Traces of
Large Random Symmetric Matrices With Independent Matrix Elements,
\textit{Bol. Soc. Bras. Mat } \textbf{29} (1998), 1--24.


%\bibitem{mehta1} M.~L.~Mehta, {\it Matrix Theory: Selected Topics and Useful Results\/},
% Hindustan Publising Corporation, 1989.


%\bibitem{mehta2} M.~L.~Mehta, {\it Random Matrices\/},
%3nd ed.,  Academic Press, San Diego, 2004.

%\bibitem{mpk} S. A. Molchanov, L. A. Pastur and A. M. Khorunzhy, Eigenvalue distribution for band random matrices
%in the limit of their infinite rank, \textit{Teor. Matem. Fizika.}
%\textbf{90} (1992), 108--118.

%\bibitem{ns} A.~Nica and R. Speicher, {\it Lectures on the Combinatorics of Free Probability\/},
%Cambridge University Press,  2006.

%\bibitem{porter} C.~E.~Porter (Ed.), {\it Statistical Theoryof Spectra: Fluctuations\/},  Academic Press, New York, 1965.
%
%\bibitem{shlyakhtenko} D. Shlyakhtenko, Random Gaussian band matrices and freeness with amalgation,
%\textit{Int. Math. Res. Not.} \textbf{20} (1996), 1013--1025.
%\bibitem{voiculescu1} D. Voiculescu,  Limit laws for random matrices and free products,
%\textit{Invent.\ Math.} \textbf{104} (1991), 201--220.
%\bibitem{voiculescu2} D. Voiculescu,  Lectures on free probability theory,
%\textit{Lecture Notes in Math. } \textbf{1738} (2000), 279--349.


%\bibitem{wigner1} E. P.~Wigner, Characteristic vectors of bordered matrices with infinite dimensions, \textit{Ann.\ Math.} \textbf{12} (1955), 548--564.
%\bibitem{wigner2} E. P.~Wigner, On the distribution of the roots of certain symmetric matrices, \textit{Ann.\ Math.} \textbf{67} (1958), 325--327.





%Central limit theorem for linear eigenvalue statistics of random
%matrices with independent entries
%
%A. Lytova and L. Pastur
%
%Source: Ann. Probab. Volume 37, Number 5 (2009), 1778-1840. Abstract
%
%We consider n¡Án real symmetric and Hermitian Wigner random matrices
%n6Ó11/2W with independent (modulo symmetry condition) entries and
%the (null) sample covariance matrices n6Ó11X*X with independent
%entries of m¡Án matrix X. Assuming first that the 4th cumulant
%(excess) ¦Ê4 of entries of W and X is zero and that their 4th
%moments satisfy a Lindeberg type condition, we prove that linear
%statistics of eigenvalues of the above matrices satisfy the central
%limit theorem (CLT) as n¡ú¡Þ, m¡ú¡Þ, m/n¡úc¡Ê[0, ¡Þ) with the same
%variance as for Gaussian matrices if the test functions of
%statistics are smooth enough (essentially of the class C5). This is
%done by using a simple ¡°interpolation trick¡± from the known
%results for the Gaussian matrices and the integration by parts,
%presented in the form of certain differentiation formulas. Then, by
%using a more elaborated version of the techniques, we prove the CLT
%in the case of nonzero excess of entries again for essentially 6»75
%test function. Here the variance of statistics contains an
%additional term proportional to ¦Ê4. The proofs of all limit
%theorems follow essentially the same scheme. Primary Subjects:
%15A52, 60F05 Secondary Subjects: 62H99 Keywords: Random matrix;
%linear eigenvalue statistics; central limit theorem
%
%Full-text: Access by subscription (subscriber: Swets Information
%Svcs Inc) Screen Optimized PDF File (589 KB) Links and Identifiers
%
%Permanent link to this document:
%http://projecteuclid.org/euclid.aop/1253539857 Digital Object
%Identifier: doi:10.1214/09-AOP452 Zentralblatt MATH identifier:
%05625053 Mathematical Reviews number (MathSciNet): MR2561434 back to
%Table of Contents References [1] Akhiezer, N. I. and Glazman, I. M.
%(1993). Theory of Linear Operators in Hilbert Space. Dover, New
%York. Mathematical Reviews (MathSciNet): MR1255973 [2] Anderson, G.
%W. and Zeitouni, O. (2006). A CLT for a band matrix model. Probab.
%Theory Related Fields 134 283¨C338. Mathematical Reviews
%(MathSciNet): MR2222385 Zentralblatt MATH: 1084.60014 Digital Object
%Identifier: doi:10.1007/s00440-004-0422-3 [3] Arharov, L. V. (1971).
%Limit theorems for the characteristic roots of a sample covariance
%matrix. Dokl. Akad. Nauk SSSR 199 994¨C997. Mathematical Reviews
%(MathSciNet): MR309171 [4] Bai, Z. D. (1999). Methodologies in
%spectral analysis of large-dimensional random matrices, a review.
%Statist. Sinica 9 611¨C677. Mathematical Reviews (MathSciNet):
%MR1711663 Zentralblatt MATH: 0949.60077 [5] Bai, Z. D. and
%Silverstein, J. W. (2004). CLT for linear spectral statistics of
%large-dimensional sample covariance matrices. Ann. Probab. 32
%553¨C605. Mathematical Reviews (MathSciNet): MR2040792 Zentralblatt
%MATH: 1063.60022 Digital Object Identifier:
%doi:10.1214/aop/1078415845 Project Euclid: euclid.aop/1078415845 [6]
%Bogachev, V. I. (1998). Gaussian Measures. Mathematical Surveys and
%Monographs 62. Amer. Math. Soc., Providence, RI. Mathematical
%Reviews (MathSciNet): MR1642391 [7] Cabanal-Duvillard, T. (2001).
%Fluctuations de la loi empirique de grandes matrices al¨¦atoires.
%Ann. Inst. H. Poincar¨¦ Probab. Statist. 37 373¨C402. Mathematical
%Reviews (MathSciNet): MR1831988 Digital Object Identifier:
%doi:10.1016/S0246-0203(00)01071-2 [8] Chatterjee, S. (2007).
%Fluctuations of eigenvalues and second order Poincar¨¦ inequalities.
%Probab. Theory Related Fields 143 1¨C40. Mathematical Reviews
%(MathSciNet): MR2449121 Digital Object Identifier:
%doi:10.1007/s00440-007-0118-6 [9] Chatterjee, S. and Bose, A.
%(2004). A new method for bounding rates of convergence of empirical
%spectral distributions. J. Theoret. Probab. 17 1003¨C1019.
%Mathematical Reviews (MathSciNet): MR2105745 Zentralblatt MATH:
%1063.60024 Digital Object Identifier: doi:10.1007/s10959-004-0587-9
%[10] Costin, O. and Lebowitz, J. L. (1995). Gaussian fluctuations in
%random matrices. Phys. Rev. Lett. 75 69¨C72. [11] Diaconis, P. and
%Evans, S. N. (2001). Linear functionals of eigenvalues of random
%matrices. Trans. Amer. Math. Soc. 353 2615¨C2633. Mathematical
%Reviews (MathSciNet): MR1828463 Zentralblatt MATH: 1008.15013
%Digital Object Identifier: doi:10.1090/S0002-9947-01-02800-8 JSTOR:
%links.jstor.org [12] Girko, V. (2001). Theory of Stochastic
%Canonical Equations 1, 2. Kluwer, Dordrecht. [13] Guionnet, A.
%(2002). Large deviations upper bounds and central limit theorems for
%non-commutative functionals of Gaussian large random matrices. Ann.
%Inst. H. Poincar¨¦ Probab. Statist. 38 341¨C384. Mathematical
%Reviews (MathSciNet): MR1899457 Zentralblatt MATH: 0995.60028
%Digital Object Identifier: doi:10.1016/S0246-0203(01)01093-7 [14]
%Ibragimov, I. A. and Linnik, Y. V. (1971). Independent and
%Stationary Sequences of Random Variables. Wolters-Noordhoff,
%Groningen. Mathematical Reviews (MathSciNet): MR322926 Zentralblatt
%MATH: 0219.60027 [15] Johansson, K. (1998). On fluctuations of
%eigenvalues of random Hermitian matrices. Duke Math. J. 91 151¨C204.
%Mathematical Reviews (MathSciNet): MR1487983 Zentralblatt MATH:
%1039.82504 Digital Object Identifier:
%doi:10.1215/S0012-7094-98-09108-6 Project Euclid:
%euclid.dmj/1077231893 [16] Jonsson, D. (1982). Some limit theorems
%for the eigenvalues of a sample covariance matrix. J. Multivariate
%Anal. 12 1¨C38. Mathematical Reviews (MathSciNet): MR650926
%Zentralblatt MATH: 0491.62021 Digital Object Identifier:
%doi:10.1016/0047-259X(82)90080-X [17] Keating, J. P. and Snaith, N.
%C. (2000). Random matrix theory and ¦Æ(1/2+it). Comm. Math. Phys.
%214 57¨C89. Mathematical Reviews (MathSciNet): MR1794265 Digital
%Object Identifier: doi:10.1007/s002200000261 [18] Khorunzhy, A.,
%Khoruzhenko, B. and Pastur, L. (1995). On the 1/N corrections to the
%Green functions of random matrices with independent entries. J.
%Phys. A 28 L31¨CL35. Mathematical Reviews (MathSciNet): MR1325832
%Zentralblatt MATH: 0855.15016 Digital Object Identifier:
%doi:10.1088/0305-4470/28/1/006 [19] Mehta, M. L. (1991). Random
%Matrices, 2nd ed. Academic Press, Boston, MA. Mathematical Reviews
%(MathSciNet): MR1083764 [20] Marchenko, V. and Pastur, L. (1967).
%The eigenvalue distribution in some ensembles of random matrices.
%Math. USSR Sbornik 1 457¨C483. [21] Muirhead, R. J. (1982). Aspects
%of Multivariate Statistical Theory. Wiley, New York. Mathematical
%Reviews (MathSciNet): MR652932 Zentralblatt MATH: 0556.62028 [22]
%Pastur, L. A. (1972). The spectrum of random matrices. Teoret. Mat.
%Fiz. 10 102¨C112. Mathematical Reviews (MathSciNet): MR475502 [23]
%Pastur, L. A. (2005). A simple approach to the global regime of
%Gaussian ensembles of random matrices. Ukra0Š7n. Mat. Zh. 57
%790¨C817. [24] Pastur, L. (2006). Limiting laws of linear eigenvalue
%statistics for Hermitian matrix models. J. Math. Phys. 47 103303.
%Mathematical Reviews (MathSciNet): MR2268864 Zentralblatt MATH:
%1112.82022 Digital Object Identifier: doi:10.1063/1.2356796 [25]
%Pastur, L. (2007). Eigenvalue distribution of random matrices. In
%Random Media 2000 (J. Wehr, ed.) 95¨C206. Wydawnictwa ICM, Warsaw.
%[26] Pastur, L. and Shcherbina, M. (2008). Bulk universality and
%related properties of Hermitian matrix models. J. Stat. Phys. 130
%205¨C250. Mathematical Reviews (MathSciNet): MR2375744 Zentralblatt
%MATH: 1136.15015 Digital Object Identifier:
%doi:10.1007/s10955-007-9434-6 [27] Prokhorov, Y. V. and Rozanov, Y.
%A. (1969). Probability Theory. Springer, Berlin. [28] Sinai, Y. and
%Soshnikov, A. (1998). Central limit theorem for traces of large
%random symmetric matrices with independent matrix elements. Bol.
%Soc. Brasil. Mat. (N.S.) 29 1¨C24. Mathematical Reviews
%(MathSciNet): MR1620151 Zentralblatt MATH: 0912.15027 Digital Object
%Identifier: doi:10.1007/BF01245866 [29] Sina09 Y. G. and
%Soshnikov, A. B. (1998). A refinement of Wigner¡¯s semicircle law in
%a neighborhood of the spectrum edge for random symmetric matrices.
%Funct. Anal. Appl. 32 114¨C131. Mathematical Reviews (MathSciNet):
%MR1627275 [30] Soshnikov, A. (2000). The central limit theorem for
%local linear statistics in classical compact groups and related
%combinatorial identities. Ann. Probab. 28 1353¨C1370. Mathematical
%Reviews (MathSciNet): MR1797877 Zentralblatt MATH: 1021.60018
%Digital Object Identifier: doi:10.1214/aop/1019160338 Project
%Euclid: euclid.aop/1019160338 [31] Spohn, H. (1987). Interacting
%Brownian particles: A study of Dyson¡¯s model. In Hydrodynamic
%Behavior and Interacting Particle Systems (Minneapolis, Minn.,
%1986). The IMA Volumes in Mathematics and its Applications 9
%151¨C179. Springer, New York. Mathematical Reviews (MathSciNet):
%MR914981 [32] Titchmarsh, E. C. (1986). Introduction to the Theory
%of Fourier Integrals, 3rd ed. Chelsea Publishing, New York.
%Mathematical Reviews (MathSciNet): MR942661

























\end{thebibliography}
\end{document}